\documentclass[nonacm,11pt]{acmart}
\usepackage{hyperref}

\AtBeginDocument{%
  \providecommand\BibTeX{{%
    \normalfont B\kern-0.5em{\scshape i\kern-0.25em b}\kern-0.8em\TeX}}}

\usepackage{booktabs}
\newcommand{\rkcomment}[1]{\textcolor{magenta}{#1}}

\newcommand{\abcomment}[1]{\textcolor{red}{#1}}
\newcommand{\ignore}[1]{}
\newcommand{\thesisrev}[1]{{#1}}

\newcommand{\revision}[1]{#1}

\DeclareRobustCommand{\rchi}{{\mathpalette\irchi\relax}}
\newcommand{\irchi}[2]{\raisebox{\depth}{$#1\chi$}}

\usepackage{mathtools}

\newcommand\comb[2][^n]{\prescript{#1\mkern-0.5mu}{}C_{#2}}

\def\vol{\text{vol}}

\newcommand{\partialh}{\partial_h}
\newcommand{\partialg}{\partial_g}
\newcommand{\newphi}{\hat{\phi}_h}
\newcommand{\xuphi}{\tilde{\phi}_h}
\newcommand{\phih}{\phi_h}
\newcommand{\phig}{\phi_g}
\newcommand{\newrho}{\hat{\rho}_{h,}{}}
\newcommand{\rhoh}{\rho_{h,}{}}
\newcommand{\rhog}{\rho_{g,}{}}

\ignore{

\newcommand{\partialh}{\partial_h}
\newcommand{\partialg}{\partial_g}
\newcommand{\newphi}{\hat{\phi}_h}
\newcommand{\xuphi}{\tilde{\phi}_h}
\newcommand{\phih}{\phi_h}
\newcommand{\phig}{\phi_g}
\newcommand{\newrho}{\hat{\rho}_{h,}{}}
\newcommand{\rhoh}{\rho_{h,}{}}
\newcommand{\rhog}{\rho_{g,}{}}}

\newcommand{\evalAG}{\lambda_{}{}}
\newcommand{\evalLH}{\mu_{}{}}

\newcommand{\evalNLH}{\nu_{}{}}
\newcommand{\evalNLG}{\nu_{}{}}
\newcommand{\rh}{\bfr_{}{}}

\newcommand{\bfr}{\bm{r}}

\newcommand{\bfx}{\bm{x}}
\newcommand{\bfy}{\bm{y}}

\newcommand{\bfI}{\bm{I}}

\usepackage{bm}

\newcommand{\bfxi}{{\bm\xi}}

\newcommand{\bfpsi}{{\bm\psi}}

\newcommand{\calA}{{\mathcal A}}

\newcommand{\calD}{{\mathcal D}}

\newcommand{\calI}{{\mathcal I}}

\newcommand{\calL}{{\mathcal L}}

\newcommand{\calO}{{\mathcal O}}

\newcommand{\scrG}{{\mathscr G}}
\newcommand{\scrH}{{\mathscr H}}
\newcommand{\scrI}{{\mathscr I}}

\newcommand{\scrK}{{\mathscr K}}

\newcommand{\bbR}{{\mathbb R}}

\newcommand{\bbZ}{{\mathbb Z}}

\usepackage{subcaption}
\usepackage{algorithm}
\usepackage[noend]{algpseudocode}
\usepackage[utf8]{inputenc}
\usepackage{multirow}
\usepackage{color}
\usepackage{tikz}
\usepackage[scr=pxtx]{mathalfa}
\usepackage{xspace}
\usepackage{lineno}

\usetikzlibrary{shapes.geometric,arrows,shapes,positioning}
\tikzstyle{startstop} = [ellipse, minimum width=1.5cm, minimum height=1cm,text centered, draw=black, ultra thick, fill=red!30]
\tikzstyle{process} = [rectangle, minimum width=1cm, minimum height=1cm, text centered, draw=black, ultra thick, fill=orange!30]
\tikzstyle{process1} = [rectangle, minimum width=7cm, minimum height=1cm, text centered, draw=black, ultra thick, fill=orange!30]
\tikzstyle{decision} = [diamond, minimum width=2cm, minimum height=1cm, text centered, draw=black, ultra thick, fill=green!30]
\tikzstyle{arrow} = [ultra thick,->,>=stealth]
\makeatletter
\let\@authorsaddresses\@empty
\makeatother

\begin{document}
\title{A Cheeger Inequality for Hypergraphs and Its Applications}
\author{Raj Kamal}
\authornote{The majority of this work was carried out during the author’s PhD studies at IIT Delhi.}
\affiliation{%
  \institution{Higher Colleges of Technology}
  \city{Ras Al Khaimah}
  \country{United Arab Emirates.\\	raj.cse.iitd@gmail.com}}

\author{Amitabha Bagchi}
\affiliation{%
  \institution{Indian Institute of Technology Delhi}
  \city{New Delhi}
  \country{India.\\bagchi@cse.iitd.ac.in}}



\begin{abstract}
  Hypergraphs provide a natural framework for modeling higher-order relationships, but the development of spectral techniques with provable guarantees for general non-uniform hypergraphs remains challenging. Building on Banerjee's normalized adjacency matrix and Spiro's averaging-based diffusion framework, we develop a spectral framework for non-uniform hypergraphs and establish Cheeger's inequality for their conductance. A fundamental result in the spectral theory of hypergraphs asserts that, for every non-covering hypergraph, the second-smallest eigenvalue of its normalized Laplacian is at most one. This spectral characterization yields an improved Cheeger's inequality for non-covering hypergraphs, and we show that the resulting inequality is tight on both sides using cycle and cube hypergraphs. Our framework further yields higher-order Cheeger inequalities and provides theoretical guarantees for Fiedler's spectral partitioning algorithm, all in the setting of hypergraphs. Finally and most notably, we construct a new family of optimal hypergraph expanders that is tight for the Alon-Boppana bound.
\end{abstract}
\newpage
\maketitle
\thispagestyle{empty}
\pagenumbering{arabic}

\section{Introduction}
\label{sec:intro}

Zhou, Huang, and Sch\"olkopf~\cite{zhou2006learning} pointed out as far back as 2006 that there are datasets where multiway relationships cannot be represented by pairwise relationships and need to be modeled by the hypergraph formalism. Nonetheless, there has been very little progress in the development of algorithms with provable guarantees for clustering and partitioning such datasets. One of the reasons for this is that while the usual model of graphs---we refer to them as $2$-graphs---have a well developed spectral theory which has been used profitably to describe algorithms with theoretical guarantees for problems such as local clustering (c.f. Andersen, Chung and Lang~\cite{andersen2007using}), graph partitioning (c.f. the works by Spielman and Teng~\cite{spielman1996spectral,spielman2013local}) and even multi-way partitioning (c.f. Lee, Gharan and Trevisan~\cite{lee2014multiway}), this line of attack has not materialized for hypergraphs. This is because the kind of spectral theory that lends itself to partitioning-related problems has not been satisfactorily developed so far.

A related structural property that has received attention in hypergraph theory is the notion of a covering hypergraph. A hypergraph is called covering if every pair of vertices is contained in at least one hyperedge~\cite{lu2021hamiltonian}. Equivalently, its $2$-shadow is a complete graph. The covering property guarantees that every pair of vertices participates in a common higher-order interaction and has been studied extensively in connection with extremal and Hamiltonian properties of hypergraphs. For example, Lu and Wang~\cite{lu2021hamiltonian} established strong Berge-cycle properties for covering $3$-uniform hypergraphs, while Falgas-Ravry and Zhao~\cite{falgas2016codegree} investigated extremal codegree thresholds associated with covering $3$-uniform hypergraphs. These works focus on combinatorial properties of covering hypergraphs. In contrast, our work studies covering hypergraphs from a spectral perspective and proves an improved Cheeger's inequality for non-covering hypergraphs.

The attempts at developing Cheeger-like inequalities relating an isoperimetric quantity like conductance to the eigenvalues of a matrix representation, such as the Laplacian or the normalized Laplacian, have run into questions like ``What is the correct definition of conductance?'' and, even more vexingly, ``what is the correct algebraic object whose eigenvalues should be studied to characterize such a conductance?'' Even in the cases where these questions were answered satisfactorily, like in the paper by Xu and Zhou~\cite{xu2024normalized} or the work of Mulas~\cite{mulas2021cheeger}, the restriction of the results to uniform hypergraphs---hypergraphs in which each hyperedge has the same number of vertices---limits the applicability of the results provided. Nor do these works focus on the algorithmic questions that are important for computing-related problems. In this paper, we attempt to plug these gaps and outline a theory that we demonstrate can be used to analyze various algorithms on non-uniform hypergraphs.

First, we work with an adjacency matrix proposed by Banerjee~\cite{banerjee2021spectrum} that has not been used much in the literature. Second, we define a quantity called the {\em diffusion conductance}, which differs from the usual definition of hypergraph conductance. The hypergraph conductance is defined simply as the worst-case ratio of the {\em number of hyperedges} in the boundary of the cut to the volume of the smaller side, whereas the diffusion conductance takes as its numerator the {\em quantity of diffusion} that is cut off by removing the hyperedges in the boundary while maintaining the same denominator. Diffusion conductance should be viewed not as a competing notion of hypergraph expansion, but as the isoperimetric quantity naturally associated with Banerjee's normalized Laplacian. Indeed, it is exactly the conductance of the algebraically equivalent weighted graph induced by the hypergraph. This equivalence enables us to establish a Cheeger inequality for Banerjee's normalized Laplacian and, through the relationship between diffusion and standard hypergraph conductance, to derive corresponding bounds for the latter. Thus, diffusion conductance provides the precise bridge between the spectral and isoperimetric structures underlying our hypergraph framework. We show that these bounds are tight on both sides for diffusion conductance, while for standard hypergraph conductance, the upper bound is tight up to a factor of $\Upsilon_H$, where $\Upsilon_H$ denotes the rank of the hypergraph, i.e., the maximum cardinality of a hyperedge. This factor is typically a constant independent of the number of vertices in many real-world hypergraphs. Finally, leveraging these Cheeger inequalities, we construct hypergraph expanders from $2$-graph expanders.

One major benefit of working with Banerjee's adjacency matrix---an edge size-normalized version of the adjacency matrices used for hypergraphs in the past---and the diffusion conductance together is that they lead us to the definition of an edge-weighted $2$-graph which is algebraically equivalent (AE) to a hypergraph in the sense that its adjacency matrix, Laplacian and normalized Laplacian are identical to that Banerjee's adjacency matrix, Laplacian and normalized Laplacian for the hypergraph. The AE graph is a weighted version of the clique expansion of the hypergraph and has often been used in algorithmic settings in the past (e.g., the work on partitioning by Zien, Schlag, and Chan~\cite{zien1999multilevel}), but we also find theoretical uses for it. \ignore{The AE graph helps us prove higher-order Cheeger inequalities and also improve the convergence bound of the Lov{\'a}sz-Simonovits theorem for hypergraphs proved by Kamal and Bagchi~\cite{kamal2024lovasz}.} Our results have significant algorithmic implications.

It is important to clarify the scope of this correspondence. Our results concern the class of hypergraph spectral theories generated by Banerjee's pairwise adjacency matrix. Since this matrix is determined entirely by the pairwise interactions induced by hyperedges, the resulting spectral framework is equivalent to a weighted $2$-shadow and therefore does not capture higher-order structural information that is invisible to the $2$-shadow. Within this framework, however, the algebraic equivalence provides a direct route for transferring spectral and isoperimetric results from graphs to non-uniform hypergraphs.

We prove conductance bounds on the partitions given spectral hypergraph partitioning using the second eigenvector of the Laplacian, following the method of Spielman and Teng in $2$-graphs~\cite{spielman1996spectral}, and also for multiway partitioning using the method of Lee, Gharan, and Trevisan~\cite{lee2014multiway}.  \ignore{In the case of local clustering, we are able to give a running time improvement in the personalized page rank-based algorithm, especially for uniform hypergraphs, that was analyzed by Kamal and Bagchi~\cite{kamal2024lovasz}. This improvement comes from the fact that the diffusion conductance that we define is better suited for a Lov{\`a}sz-Simonovits analysis. The improvement is at least a factor of 2 and can be as high as the size of the largest hyperedge. This can be a very significant speedup: for example, the DBPedia-genre\_LCC dataset used by Takai, Miyauchi, Ikeda, and Yoshida~\cite{takai2020hypergraph} puts musical artists from a single genre into a hyperedge and reports an average hyperedge size of 92. }

Explicit constructions of $2$-graph expanders are well studied; we refer the reader to the survey by Hoory, Linial, and Wigderson~\cite{hoory2006expander} and to Trevisan’s book~\cite{trevisan2017expanders}. In contrast, only a small number of hypergraph expander constructions are known~\cite{feng1996spectra,song2023hypergraph}. Related work includes spectral sparsification of hypergraphs, such as the recent results of Kapralov, Krauthgamer, Tardos, and Yoshida~\cite{kapralov2021towards}.

Simplicial complexes form a special class of hypergraphs that are closed under inclusion~\cite{zhang2023higher} and have been extensively studied under the name \emph{high-dimensional expanders}~\cite{lubotzky2018high}. However, many real-world networks are naturally modeled as hypergraphs that are not simplicial complexes. For example, in a network of international cricket teams, each team consists of exactly 11 players, but no proper subset forms a team. This highlights the need for further study of general hypergraph expanders. In this direction, we introduce a new class of hypergraph expanders that is tight for the Alon--Boppana bound proved by Feng and Li~\cite{feng1996spectra}.

The key contributions of our paper are
\begin{itemize}
\item When vertex weights are independent of the incident hyperedges, Chitra and Raphael showed that a random walk on a hypergraph is equivalent to a random walk on an appropriately weighted clique expansion~\cite{chitra2019random}. Building on this connection, we show that Banerjee's normalized adjacency matrix of a hypergraph is algebraically equivalent to that of a weighted clique-expansion $2$-graph. This provides a clean and unifying framework for extending powerful spectral tools from $2$-graphs to non-uniform hypergraphs.

\item We introduce a new quantity, termed diffusion conductance, which enables us to extend powerful spectral results from $2$-graphs to non-uniform hypergraphs. Using this quantity, we establish Cheeger inequalities for hypergraphs and, in particular, obtain a sharper upper bound on the conductance of hypergraphs.

\item A fundamental result in the spectral theory of hypergraphs asserts that the second-smallest eigenvalue of the normalized Laplacian of every non-covering hypergraph is at most one. This spectral characterization yields a strengthened Cheeger's inequality for non-covering hypergraphs. Moreover, we establish that the resulting bounds are tight on both sides by exhibiting cycle and cube hypergraphs that attain the corresponding bounds.

\item We translate our theoretical spectral bounds into concrete performance and running-time guarantees for $2$-way spectral partitioning and higher-order multi-way partitioning in the hypergraph setting.

\item Finally, and most notably, we construct a new family of optimal hypergraph expanders that attain the Alon--Boppana bound for hypergraphs established by Feng and Li~\cite{feng1996spectra}.

\ignore{\item Improving upon the Lov{\'a}sz-Simonovits theorem for hypergraphs, first presented by Kamal and Bagchi~\cite{kamal2024lovasz} for uniform hypergraphs, we are able to remove the dependence of the convergence time of diffusion on the size of the hyperedges. }
\end{itemize}

\paragraph{Paper organization.} In Section~\ref{sec:prelims}, we present our notation and definitions and foundational results from $2$-graphs, and also, in Section~\ref{sec:prelims:ae-graph}, we define the algebraically equivalent (AE) $2$-graph of a hypergraph and prove some simple properties of the construction. Our Cheeger's inequality for hypergraphs is presented in Section~\ref{section:cheeger-inequality}, with Section~\ref{section:cheeger-inequality:tightness} containing a discussion on tight examples and Section~\ref{section:cheeger-inequality-high-order} presenting higher order Cheeger inequalities. \ignore{Section~\ref{sec:ls-theorem} describes our improved Lov\'asz-Simonovits theorem for hypergraphs.} The algorithmic implications of our results are discussed in detail in Section~\ref{sec:algorithmic}. We present our hypergraph expander construction results in Section~\ref{section:expander-hypergraphs}. A discussion of related works from the literature is carried out in Section~\ref{sec:related}. We pose some open questions and discuss the significance of our results in Section~\ref{sec:conclusion}.

\section{Related Work}
\label{sec:related}

The algebraic study of hypergraphs is either through matrices based on the adjacency structure starting from Rodr\'iguez~\cite{rodriguez:2002} or using hypermatrices and tensor algebra (c.f., Banerjee, Char and Mondal~\cite{banerjee2017spectra}). Since the latter approach is restricted to uniform hypergraphs, we do not explore that direction further. Banerjee introduced a version of the adjacency matrix that had entries normalized by one less than the edge size for each edge~\cite{banerjee2021spectrum}, considering it more suitable for use in non-uniform hypergraphs. We have also used this matrix in our work. Mulas, Kuehn, B{\"o}hle and Jost discuss hypergraph adjacency matrices in more general settings~\cite{mulas2022random}. Bellaachia and Al-Dhelaan introduce the adjacency matrix of hypergraphs with edge-dependent vertex weights~\cite{bellaachia2013random}. In all these works, the adjacency matrix is a starting point for the definition of a Laplacian or normalized Laplacian matrix, but in Chitra and Raphael~\cite{chitra2019random}, a Laplacian is defined based on a random walk defined on edge and vertex-weighted hypergraphs. Chan, Louis, Tang, and Zhang~\cite{chan2018spectral} and Takai, Miyauchi, Ikeda, and Yoshida~\cite{takai2020hypergraph} deviate from both the matrix and tensor-based routes by using the name hypergraph Laplacian for two different families of operators that vary based on the vector they are applied to. 

A structural property of hypergraphs that has been studied extensively in combinatorics is the covering property. A hypergraph is called covering if every pair of vertices is contained in at least one hyperedge~\cite{lu2021hamiltonian}. Equivalently, the $2$-shadow of the hypergraph is a complete graph. Covering hypergraphs have been investigated in connection with extremal, codegree, and Hamiltonian problems. Lu and Wang~\cite{lu2021hamiltonian} studied Hamiltonian Berge cycles in covering $3$-uniform hypergraphs, while Falgas-Ravry and Zhao~\cite{falgas2016codegree} characterized codegree thresholds for covering $3$-uniform hypergraphs. Related covering and tiling problems for uniform hypergraphs have also been studied through tight cycles and minimum codegree conditions~\cite{han2015minimum,han2021covering}. These works focus primarily on combinatorial properties of covering hypergraphs, whereas our work investigates covering hypergraphs from a spectral perspective and proves an improved Cheeger's inequality for non-covering hypergraphs.

The conductance of the 2-graph is an isoperimetric quantity, which is a discrete analog of a quantity studied by Cheeger~\cite{cheeger1970lower} in the context of Riemannian manifolds and is also known as Cheeger's constant in the literature. For hypergraphs, a simple generalization was used by Chan, Louis, Tang, and Zhang~\cite{chan2018spectral}, who defined the conductance of a cluster as the ratio of the sum of the weights of hyperedges leaving the cluster to the sum of degrees of the vertices in the cluster, i.e., the volume of the cluster. Banerjee~\cite{banerjee2021spectrum}, Takai, Miyauchi, Ikeda, and Yoshida~\cite{takai2020hypergraph}, and Kamal and Bagchi~\cite{kamal2024lovasz} also worked with this definition of conductance when discussing hypergraph clustering. Mulas~\cite{mulas2021cheeger} worked with a different quantity: the ratio of the sum of $|S\cap e||S^c\cap e|$ over all edges to the volume of $S$ was called Cheeger's constant associated with $S \subseteq V$.  Our definition of diffusion conductance (Section~\ref{sec:prelims:definitions:hypergraph}) is an edge-by-edge scaled version of this quantity.

For $2$-graphs, a theorem relating the conductance to the second smallest eigenvalue of the normalized Laplacian matrix was proved by Alon and Milman~\cite{alon1985lambda1} and Alon~\cite{alon1986eigenvalues} and given the name Cheeger's inequality. Chung~\cite{chung2007four} has provided multiple different proofs for this foundational result. In the context of multi-way-partitioning of the vertex set, Cheeger's inequalities relating higher order conductances to the higher order eigenvalues of the normalized Laplacian were established by Lee, Gharan, and Trevisan~\cite{lee2014multiway} and Kwok, Lau, Lee, Oveis Gharan, and Trevisan~\cite{kwok2017improved}.

Mulas~\cite{mulas2021cheeger} and Xu and Zhou~\cite{xu2024normalized} established Cheeger inequalities for uniform hypergraphs that are closely related to our main result, but their results are restricted to uniform hypergraphs. Chan, Louis, Tang, and Zhang established a Cheeger's inequality for hypergraphs based on a nonlinear diffusion operator~\cite{chan2018spectral}. Their diffusion operator is state-dependent rather than a fixed matrix, and their framework does not admit the usual higher-order eigenvalues. Consequently, their spectral approach differs from the matrix-based framework considered here. Banerjee established a Cheeger's inequality for general, including non-uniform, hypergraphs~\cite{banerjee2021spectrum}, but the resulting bounds are weaker than ours. Mulas and Zhang established a Cheeger's inequality for oriented hypergraphs~\cite{mulas2021spectral}. Chitra and Raphael established a Cheeger's inequality for a random walk on hypergraphs with edge-dependent vertex weights~\cite{chitra2019random}. Their Cheeger constant is the conductance of the associated random walk. Our analysis instead formulates the diffusion directly through the averaging operator and develops a Cheeger's inequality within the normalized Laplacian framework.

Spielman and Teng described a local method based on random walks for local clustering with bounded cluster conductance and running time based on conductance~\cite{spielman2013local}. Andersen, Chung, and Lang described a method for the same problem using personalized page rank~\cite{andersen2007using}. The quality of local clusters produced in both these approaches is guaranteed by Lov\'asz-Simonovits theory~\cite{lovasz1990mixing,lovasz1993random}. A program similar to the latter work was executed for hypergraphs by Kamal and Bagchi~\cite{kamal2024lovasz} who used personalized page rank derived from an averaging-based linear diffusion introduced by Spiro~\cite{spiro2022averaging}. Takai, Miyauchi, Ikeda, and Yoshida~\cite{takai2020hypergraph} gave a similar local clustering algorithm based on a non-linear diffusion operator, which requires solving a differential equation to compute and was thus computationally demanding.

The algorithm for $2$-graph partitioning using the eigenvector corresponding to the second eigenvalue, which we have adapted for hypergraphs, was described and analyzed by Mihail~\cite{mihail1989conductance} and Spielman and Teng~\cite{spielman1996spectral}. The area of $2$-graph partitioning is vast, and so we don't attempt to survey it here, referring the reader to the book by Trevisan~\cite{trevisan2017expanders}. Trevisan's book is also a good reference for the multiway spectral partitioning algorithm described by Lee, Gharan, and Trevisan~\cite{lee2014multiway}, which we show can easily be adapted to hypergraphs. Although hypergraph partitioning has been studied since the 1990s due to its applications in VLSI design (c.f., e.g., the work of Karypis, Aggarwal, Kumar and Shekhar~\cite{karypis1997multilevel}), mostly the partitioning methods are not spectral in nature and work on a $2$-graph transformation of the hypergraph: either the clique expansion (c.f., e.g., Zien, Schlag and Chand~\cite{zien1999multilevel}) or the star expansion which creates a new vertex for each hyperedge and connects it to all the vertices that belong to that edge (c.f, e.g., Zhou, Huang and Sch\"olkopf~\cite{zhou2006learning}.)

Several explicit constructions of $2$-graph expanders are known in the literature. Instead of listing them here, we refer the readers to the survey by Hoory, Linial, and Wigderson~\cite{hoory2006expander} and Trevisan's book~\cite{trevisan2017expanders}. To the best of our knowledge, there are only a few constructions for hypergraph expanders~\cite{feng1996spectra,song2023hypergraph}, and there has also been some work on the sparsification of hypergraphs that relies on spectral theory; recent work in this line being that of Kapralov, Krauthgamer, Tardos, and Yoshida~\cite{kapralov2021towards}. In addition, a simplicial complex is a particular case of a hypergraph that is closed under inclusion. In other words, a hypergraph $H$ is said to be a simplicial complex if $e$ is a hyperedge in $H$, then every subset of $e$ is also a hyperedge~\cite{zhang2023higher}. The simplicial complex expanders have been extensively studied in the literature and are called high-dimensional expanders (see, e.g.,~\cite{lubotzky2018high}). But there are many networks where a set of nodes forms a hyperedge, and not all proper subsets of this hyperedge are a hyperedge. For example, consider the network of international cricket teams. Each team contains exactly $11$ players, but any group of fewer than $11$ players does not form a cricket team. Therefore, the study of hypergraph expanders is still incomplete in general. In this line of research, we add one hypergraph expander class tight for the Alon-Boppana bound proved by Feng and Li~\cite{feng1996spectra}.
\section{Preliminaries}
\label{sec:prelims}

\subsection{Definitions}
\label{sec:prelims:definitions}

\paragraph{Hypergraphs and conductance(s)}
\label{sec:prelims:definitions:hypergraph}
A {\em hypergraph} $H$ is a tuple $H=(V_H, E_H, w_H)$ where $V_H$ is a set of $n$ vertices and $E_H \subseteq 2^V$ is a set of $m$ {\em hyperedges} where we do not allow hyperedges of size one (also called self-loops). For convenience, we sometimes use the notation $V(H)$ and $E(H)$ to denote the vertex and hyperedge sets of $H$, respectively. A hypergraph is called {\em covering} if every pair of distinct vertices is contained in at least one hyperedge. A hypergraph that is not covering is called a {\em non-covering} hypergraph. The weight function $w_H:E_H\longrightarrow [0,\infty)$ assigns a unique weight $w_H(e)$ to each hyperedge $e\in E_H$. The maximum cardinality of a hyperedge of $E$ is referred to as the {\em rank} of $H$, and it is denoted by $\Upsilon_H$. The minimum cardinality of a hyperedge of $E$ is referred to as the {\em co-rank} of $H$, and it is denoted by $\kappa_H$. The usual definition of graphs (without self-loops) can be thought of as hypergraphs of $\Upsilon_H = \kappa_H = 2$ and is referred to as {\em 2-graphs} in this paper. For clarity, we use the letter $G$ for 2-graphs in this paper. If all the hyperedges of $E$ have the same cardinality, $\kappa$, we say that $H$ is {\em $\kappa$-uniform}. The degree of a vertex $u$, denoted $d_H(u)$, is the sum of the weights of all hyperedges containing $u$. A hypergraph where each vertex has the same degree, $d$, is called $d$-regular. A hypergraph is said to be {\em connected} if for every $u, v \in V$, there exists a sequence of hyperedges $e_1,e_2, \ldots, e_{\ell}$ such that $u \in e_1, v\in e_{\ell}$ and $e_i \cap e_{i+1} \ne \emptyset$ for $1 \leq i < \ell$. If all hyperedges are distinct in the above sequence, we say that the sequence $e_1,e_2, \ldots, e_{\ell}$ is a path of length $\ell-2$ connecting hyperedges $e_1,~e_{\ell}$ and the hyperedges $e_1,~e_{\ell}$ have distance $\ell-2$ between them along this path. The length of the shortest path connecting hyperedges $e_1$ and $e_{\ell}$ is called the distance between the hyperedges $e_1$ and $e_{\ell}$.

The {\em volume} of $S \subseteq V_H$ is defined by $\vol_H(S)\coloneqq\sum_{v:v\in S}d_H(v)$. The {\em edge boundary} of the set $S$ is \[\partialh(S)\coloneqq\left\{e\in E:e\cap S\neq\emptyset, e\cap V_H\setminus S \neq\emptyset\right\}.\]
For clarity, when $G$ is a $2$-graph, we use $\partialg(S)$ to denote the edge boundary. The {\em conductance} of $S$ is defined by \[\phih(S)\coloneqq \frac{\sum_{e:e\in \partialh(S)}w_H(e)}{\min(\vol_H(S),\vol_H(V_H\setminus S))}.\] The conductance of the hypergraph $H$ is defined as $\phih(H)=\min_{S:\emptyset\subsetneq S\subsetneq V_H} \phih\left(S\right)$. The conductance is also referred to as the {\em Cheeger constant} in the literature (cf., e.g., Banerjee~\cite{banerjee2021spectrum}.) As before, if the graph $G$ is a 2-graph, we use $\phig(S)$ to denote the conductance of a set $S$ and $\phig(G)$ to denote the conductance of the 2-graph. 
We define a new quantity called the {\em diffusion conductance} of $S$ as
\[\newphi (S)\coloneqq \frac{\sum_{e:e\in \partialh(S)}\frac{w_H(e)}{|e|-1}|e\cap S||e\cap (V_H\setminus S)|}{\min(\vol_H(S),\vol_H(V_H\setminus S))}.\]
The diffusion conductance of $H$ is defined by $\newphi(H)=\min_{S:\emptyset\subsetneq S\subsetneq V_H} \newphi\left(S\right)$.
We note that this $\newphi(H)$ and $\phih(H)$ are different from each other, but both give the same value when $\Upsilon_H = \kappa_H = 2$, i.e., when $H$ is actually a 2-graph. That is to say, they are two different generalizations for the 2-graph conductance when we try to extend it to hypergraphs. These two quantities, while different, are closely related.
For hypergraphs with rank $\Upsilon_H$, it is easy to see that
\begin{equation}
\label{eq:banerjee-S}
   \phih(S) \leq \newphi(S) \leq  \frac{\Upsilon_H}{2} \phih(S).
\end{equation}
for all $S:\emptyset\subsetneq S\subsetneq V_H$ which implies that
\begin{equation}
\label{eq:banerjee}
   \phih(H) \leq \newphi(H) \leq  \frac{\Upsilon_H}{2} \phih(H).
\end{equation}
In Section~\ref{section:cheeger-inequality:tightness:lower}, we show that the upper bound is attained by the cube hypergraph, whereas the lower bound is attained by the bicycle hypergraph (see Section~\ref{section:cheeger-inequality-tightness-upper-bicycle}).

\paragraph{Algebraic definitions}
\label{sec:prelims:definitions:algebraic}
Banerjee~\cite{banerjee2021spectrum} suggested the following adjacency matrix for hypergraphs:
\begin{equation}
\label{adjacency:hypergraph}
A_H[u,v]=\sum_{e\in E_H:\{u,v\}\subseteq e}\frac{w_H(e)}{|e|-1}
\end{equation}
when $u\neq v$, and $0$ otherwise. Note that this adjacency matrix reduces to the standard adjacency matrix $A_G$ for a $2$-graph $G$. We denote eigenvalues of $A_H$ by $\lambda_1(A_H) \geq \lambda_2(A_H) \geq \cdots \geq \lambda_n(A_H)$ sorted in descending order. Let $D_H$ be the diagonal degree matrix with $D_H[u,u]=d_H(u)$ for all $u\in V$.
Also, let the matrix $\calA_H$ be defined by
\begin{equation}
\label{diffusion:hypergraph}
\calA_H[u,v]=\sum_{e\in E_H:\{u,v\}\subseteq e}\frac{w_H(e)}{|e|}
\end{equation}
where $u\text{ and }v$ are not restricted to be distinct. Then, the {\em lazy diffusion} matrix is defined by $\calD_H=\calA_HD_H^{-1}$. Note that $\calD_H$ is a generalization of the lazy random walk in 2-graphs since $|e| = 2$ in that case. Let the Laplacian of $H$ be defined by
$L_H=D_H-A_H$. Note that the quadratic form associated with $L_H$ can be written as:
\begin{align}
  \bfx^TL_H\bfx&=\sum_{e:e\in E}\frac{w_H(e)}{|e|-1}\sum_{\{u,v\}:\{u,v\}\subseteq e}(\bfx(u)-\bfx(v))^2 \label{eq:lap-quad-form}\\
  &=\sum_{e:e\in E}\frac{w_H(e)\times|e|}{|e|-1}\sum_{u\in e}(\bfx(u)-\overline{\bfx}_e)^2
\end{align}
where $\overline{\bfx}_e=\frac{\sum_{v\in e}\bfx(v)}{|e|}$. Therefore, the quadratic form of the hypergraph Laplacian measures how different the values of a vector are across vertices in the same hyperedge. Let $\evalLH_1(H) \leq \evalLH_2(H) \leq \cdots \leq \evalLH_n(H)$ be the eigenvalues of $L_H$ sorted in ascending order. From~\eqref{eq:lap-quad-form} it is clear that $L_H$ is positive semidefinite and the vector $\bm{1} = (1, 1,\ldots, 1) \in \bbR^n$ with all 1 entries is an eigenvector corresponding to $\evalLH_1(H) = 0$.
The normalized Laplacian $\calL_H$ of $H$ is defined by $\calL_H=D_H^{-1/2}L_HD_H^{-1/2}$. Note that
\[\frac{\bfx^T\calL_H\bfx}{\bfx^T\bfx}=\frac{\bfy^TL_H\bfy}{\bfy^TD_H\bfy}\]
where $\bfx=D_H^{1/2}\bfy$. We denote by $\evalNLH_1(H) \leq \evalNLH_2(H) \leq \cdots \leq \evalNLH_n(H)$ the eigenvalues of $\calL_H$ in ascending order. Let $\bfpsi_1$ be the vector such that $\bfpsi_1(u)=\sqrt{d_H(u)}$, and let $\hat{\bfpsi}_1$ be the vector such that $\hat{\bfpsi}_1(u)=d_H(u)$. Then, $\bfpsi_1$ is eigenvector of $\calL_H$ corresponding to eigenvalue $0$ and
\begin{equation}
  \label{eq:nu2}
\evalNLH_2(H)= \min_{\bfx\perp \bfpsi_1}\frac{\bfx^T\calL_H\bfx}{\bfx^T\bfx}=\min_{\bfy\perp \hat{\bfpsi}_1}\frac{\bfy^TL_H\bfy}{\bfy^TD_H\bfy}.
\end{equation}
\ignore{We replace the $h$ in the subscript with a $g$ both for $\mu$ and $\nu$ when we are dealing with a 2-graph. The use of the letter in the subscript makes the notation heavy, but it is useful when we are comparing hypergraphs and 2-graphs. The notation used in this chapter is given in Table~\ref{table:notations}.}

The notation used in this paper is given in Table~\ref{table:notations}.

\begin{table}[htbp]
    \centering
    \captionsetup{labelfont={color=black}} 
    \caption{Notations}
    \label{table:notations}
    \revision{
    \begin{tabular*}{\columnwidth}{@{\extracolsep{\fill}}p{0.25\columnwidth}p{0.7\columnwidth}}
        \toprule
        \textbf{Symbol} & \textbf{Meaning} \\
        \midrule
        $H=(V_H,E_H,w_H)$ & A hypergraph with vertex set $V_H$, hyperedge set $E_H$, and weight function $w_H:E_H\longrightarrow [0,\infty)$ assigning a unique weight to each hyperedge.\\
        $G_H=(V_{G_H},E_{G_H},w_{G_H})$ & AE $2$-graph of $H$.\\
        $\Upsilon_H$ & Rank of hypergraph $H$ representing size of the largest hyperedge of $H$.\\
        $\kappa_H$ & Co-rank of hypergraph $H$ representing size of the smallest hyperedge of $H$.\\
        $d_H(v)$ & Degree of vertex $v$ i.e. sum of the weights of all hyperedges containing vertex $v$.\\
        $\vol_H(S)$ & Volume of vertex set $S$, i.e., sum of the degrees of vertices in $S$.\\
        $\Delta(H)$ & Sum of the degrees of all vertices of $H$ i.e. $\vol_H(V_H)$.\\
        $\partial_h(S)$ & Edge boundary of vertex set $S$ i.e. number of hyperedges having non-empty intersection with both $S$ and $V_H\setminus S$.\\
        $\phih(S)$ & Conductance of vertex set $S$.\\
        $\phih(H)$ & Conductance of hypergraph $H$.\\
        $\newphi(S)$ & Diffusion conductance of vertex set $S$.\\
        $\newphi(H)$ & Diffusion conductance of hypergraph $H$.\\
        $A_H$ & Adjacency matrix of hypergraph $H$.\\
        $D_H$ &  Degree diagonal matrix of hypergraph $H$.\\
        $L_H$ & Lapalacian matrix $D_H-A_H$ of hypergraph $H$.\\
        $\calL_H$ & Normalized Laplacian $D_H^{-1/2}L_HD_H^{-1/2}$ of hypergraph $H$.\\
        $\evalLH_{\ell}$ & $\ell$-th smallest eigenvalue of Laplacian $L_H$.\\
        $\evalNLH_{\ell}$ & $\ell$-th smallest eigenvalue of normalized Laplacian $\calL_H$.\\
        $\rhoh_{\ell}$ & $\ell$-th order expansion.\\
        $\newrho_{\ell}$ & $\ell$-th order diffusion expansion.\\
        $H_{\ell,k}$ & Cube hypergraph defined in Section~\ref{section:cheeger-inequality:tightness:lower}.\\
        $C_{n,k}$ & Cycle hypergraph defined in Section~\ref{section:cheeger-inequality:tightness:upper}.\\
        $C_{2n,k}$ & Bicycle hypergraph defined in Section~\ref{section:cheeger-inequality-tightness-upper-bicycle}.\\
\ignore{        $I_h^{(\ell)}$ & Lov\`asz-Simonovits curve, after $\ell$-th iteration of the diffusion process on hypergraphs, defined in Section~\ref{sec:ls-theorem}.\\
        $I_g^{(\ell)}$ & Lov\`asz-Simonovits curve, after $\ell$-th iteration of the diffusion process on $2$-graphs, defined in Section~\ref{sec:ls-theorem-proof}.\\}
        \bottomrule
    \end{tabular*}}
    \vspace{-0.2in}
\end{table}

\subsection{Cheeger Inequalities for \texorpdfstring{$2$}{2}-graphs}
\label{sec:prelims:cheeger}
We state Cheeger's inequality for $2$-graphs (Alon and Milman~\cite{alon1985lambda1}; Alon~\cite{alon1986eigenvalues}). We will use this inequality for deriving Cheeger's inequality for hypergraphs in Section~\ref{section:cheeger-inequality}.
\begin{theorem}[Cheeger's inequality for $2$-graphs]
\label{thm:cheeger-ineq-graph}
If $G=(V_G,E_G.w_G)$ is a connected $2$-graph and $\evalNLG_2(G)$ is the second smallest eigenvalue of its normalized Laplacian $\calL_G$ then
\begin{equation}
  \label{eq:cheeger-graph}
  \frac{\evalNLG_2(G)}{2} \leq \phig(G) \leq \sqrt{2\evalNLG_2(G)}
\end{equation}
\end{theorem}
Lee, Gharan, and Trevisan have shown a family of results known as higher-order Cheeger inequalities~\cite{lee2014multiway}. To state these, we first define the higher-order expansion, an extension of conductance.

Given a hypergraph $H=(V_H,E_H,w_H)$ and $1 \leq \ell \leq n$\ignore{\abcomment{RK, please check}}, let $\emptyset\subsetneq S_1,~S_2,~\ldots,~S_{\ell}\subsetneq V_H$ be a collection of ${\ell}$ disjoint subsets of vertices. Then, we define the {\em order-${\ell}$ expansion} of this collection as
\[\rhoh_{\ell}(S_1,~S_2,~\ldots,~S_{\ell})=\underset{i\in\{1,2,\ldots,{\ell}\}}{\max}\phih(S_i),\]
and the order-${\ell}$ expansion of $H$ is defined by $\rhoh_{\ell}(H)=\underset{T_1,T_2,\ldots,T_{\ell}}{\min}\rhoh_{\ell}(T_1,T_2,\ldots,T_{\ell})$ where $\emptyset\subsetneq T_1,T_2,\ldots,T_{\ell}\subsetneq V_H$ are disjoint subsets of vertices. As before, when we know that the graph $G$ is a 2-graph, we will replace the $h$ in the subscript with a $g$ to get $\rhog_{\ell}(G)$.
We will also define the {\em order-$\ell$ diffusion expansion} in exactly the same way with $\newphi$ used in place of $\phih$. The notation for this will be $\newrho_{\ell}(H)$. For hypergraphs with rank $\Upsilon_H$, Eq.~\eqref{eq:banerjee} implies that
\begin{equation}
\label{eq:trevisan}
   \rhoh_{\ell}(H) \leq \newrho_{\ell}(H) \leq  \frac{\Upsilon_H}{2} \rhoh_{\ell}(H).
\end{equation}
Lee, Gharan, and Trevisan proved the following higher-order Cheeger's inequality for $2$-graphs~\cite{lee2014multiway}:
\begin{theorem}[Higher Order Cheeger's inequality for $2$-graphs~\cite{lee2014multiway}]
\label{thm:cheeger-ineq-graph-higher}
If $G=(V_G,E_G,w_G)$ is any $2$-graph and $\evalNLG_{\ell}(G)$ is the ${\ell}$-th smallest eigenvalue of its normalized Laplacian $\calL_G$ then
\begin{equation}
  \label{eq:cheeger-graph-high-order}
  \frac{\evalNLG_{\ell}(G)}{2} \leq \rhog_{\ell}(G) \leq \calO({\ell}^2)\sqrt{\evalNLG_{\ell}(G)}
\end{equation}
\end{theorem}
\subsection{An Algebraically Equivalent \texorpdfstring{$2$}{2}-Graph for Hypergraphs}
\label{sec:prelims:ae-graph}
When vertex weights are independent of the incident hyperedges, the weight of a vertex remains the same across all hyperedges containing it. Chitra and Raphael showed that, under this condition, a random walk on a hypergraph is equivalent to a random walk on an appropriately weighted clique expansion~\cite{chitra2019random}. For completeness, we explicitly formulate this equivalence for averaging-based diffusion and use it to establish a Cheeger's inequality for hypergraphs.

In particular, for any hypergraph $H$, it is possible to construct a weighted $2$-graph $G_H$ that is algebraically equivalent to $H$, in the sense that $H$ and $G_H$ have exactly the same adjacency matrices. We will refer to $G_H$ as the {\em algebraically equivalent 2-graph of $H$}, or, for brevity, the {\em AE graph of $H$}.
\begin{lemma}\label{lemma:graph-equivalence}
Given a hypergraph $H=(V_H,E_H,w_H)$, we can construct a $2$-graph $G_H=(V_{G_H},E_{G_H},w_{G_H})$ such that hypergraph $H$ and $2$-graph $G_H$ have identical adjacency, Laplacian, and normalized Laplacian matrices.
\end{lemma}
\begin{proof}
    Let $G_H$ be the $2$-graph obtained from hypergraph $H$ as follows. The vertex set of $G_H$ is the same as that of $H$, i.e., $V_{G_H}=V_H$. Each hyperedge $e$ of $H$ is replaced by a clique connecting all the vertices of $e$ with all edge weights $\frac{w_H(e)}{|e|-1}$. Let $u,v\in V_H$ be any two distinct vertices. For every hyperedge $e$ containing vertices $u$ and $v$ in $H$, an edge with weight $\frac{w_H(e)}{|e|-1}$ connects these vertices in $G_H$. Therefore, the adjacency matrix $A_{G_H}$ of the $2$-graph $G_H$ is given by
\[A_{G_H}[u,v]=\sum_{e:\{u,v\}\subseteq e}\frac{w_H(e)}{|e|-1}\]
Since any vertex pair may be contained in multiple hyperedges in $H$. So, we allow multiple edges in $G_H$ to connect the same vertex pair (equivalently, we just increase the weight of the edge connecting this pair by adding the weights coming from all the hyperedges containing this pair\ignore{\abcomment{RK, please check.})}.
Using this with Eq.~\eqref{adjacency:hypergraph}, we get $A_{G_H}[u,v]=A_H[u,v]$. Therefore, the adjacency matrices of $H$ and $G_H$ are identical. It implies that $H$ and $G_H$ have identical Laplacian and normalized Laplacian matrices.
\end{proof}
Note that the AE-graph of $H$ is simply a weighted version of the clique expansion of $H$. Next, we discuss the key property of the AE-graph of $H$: its conductance $\phig(G_H)$ is equal to the diffusion conductance $\newphi(H)$ of $H$. 
\begin{lemma}\label{lemma:conductance-equivalence}
    Given a hypergraph $H=(V_H,E_H,w_H)$, let $G_H$ be the AE-graph of $H$. If $T:\emptyset\subsetneq T\subsetneq V_H$ is any subset of vertices, then $\phig(T)=\newphi(T)$.
\end{lemma}
\begin{proof} We have
\[
    \phig(T) =\frac{\sum_{e:e\in\partial_h(T)}\sum_{\{u,v\}\in\partial_g(T):\{u,v\}\subseteq e}\frac{w_H(e)}{|e|-1}}{\min(\vol_{G_H}(T),\vol_{G_H}(V_{G_H}\setminus T))} =\frac{\sum_{e:e\in \partialh(T)}\frac{w_H(e)}{|e|-1}|e\cap T||e\cap (V_H\setminus T)|}{\min(\vol_H(T),\vol_H(V_H\setminus T))} =\newphi(T)
\]
\end{proof}
This property naturally transfers to the order-$\ell$ expansion of $H$ 
\begin{lemma}\label{lemma:high-order-conductance-equivalence}
    Given a hypergraph $H=(V_H,E_H,w_H)$, let $G_H$ be the AE-graph of $H$.  If $\emptyset\subsetneq S_1,~S_2,~\ldots,~S_\ell\subsetneq V_H$ are disjoint subsets of vertices then, for $1 \leq \ell \leq n$\ignore{\abcomment{RK, please check}} $$\rhog_\ell(S_1,~S_2,~\ldots,~S_\ell)=\newrho_\ell(S_1,~S_2,~\ldots,~S_\ell).$$
\end{lemma}
The proof easily follows from Lemma~\ref{lemma:conductance-equivalence}.

\section{A Cheeger's Inequality for Hypergraphs}
\label{section:cheeger-inequality}

In this section, we establish Cheeger inequalities for hypergraphs. Chung proved $\phig(G) \leq 2\sqrt{(2-\evalNLH_2(G))\evalNLH_2(G)}$ for a general graph $G$~\cite{chung1996laplacians}. Banerjee extended this type of Cheeger's inequality to general hypergraphs, obtaining $\phih(H) \leq (\Upsilon_H-1)\sqrt{(2-\evalNLH_2(H))\evalNLH_2(H)}$ for a connected hypergraph $H$~\cite{banerjee2021spectrum}. We first improve Banerjee's bound to $\phih(H) \leq \sqrt{2\evalNLH_2(H)}$. We then use a spectral characterization of non-covering hypergraphs to derive a further-improved Cheeger's inequality relating the second-smallest eigenvalue of the normalized Laplacian to the conductance of non-covering hypergraphs. We show in Section~\ref{section:cheeger-inequality:tightness} that this inequality is tight.

\begin{theorem}[Cheeger's inequality for hypergraphs]
\label{thm:cheeger-ineq-general}
Let $H$ be a connected hypergraph and let $\evalNLH_2(H)$ be the second smallest eigenvalue of its normalized Laplacian $\calL_H$. Then, we have
\begin{equation}
  \label{eq:cheeger}
  \frac{\evalNLH_2(H)}{\Upsilon_H} \leq \phih(H) \leq \sqrt{2\evalNLH_2(H)}.
\end{equation}
where $\Upsilon_H$ is rank of $H$.
\end{theorem}
\begin{proof}The Cheeger's Inequality \eqref{eq:cheeger-graph} for $2$-graphs and Lemma~\ref{lemma:graph-equivalence} imply that
\begin{equation}
  \label{eq:cheeger-ineq-2-graph}
  \frac{\evalNLH_2(H)}{2} \leq \phig(G_H) \leq \sqrt{2\evalNLH_2(H)}
\end{equation}
where $G_H$ is the AE graph of $H$. Lemma~\ref{lemma:conductance-equivalence} implies that $\phig(G_H)=\newphi(H)$. Therefore, we have
\begin{equation}
  \label{eq:cheeger-ineq-2-graph-diffusion}
  \frac{\evalNLH_2(H)}{2} \leq \newphi(H) \leq \sqrt{2\evalNLH_2(H)}
\end{equation}

Using Eq.~(\ref{eq:banerjee}), it implies the following
\[\frac{\evalNLH_2(H)}{\Upsilon_H} \leq \phih(H) \leq \sqrt{2\evalNLH_2(H)}\]
which proves the result.
\end{proof}

Chan, Louis, Tang, and Zhang established the Cheeger's inequality $\frac{\gamma_2}{2}\leq\phih(H)\leq\sqrt{2\gamma_2}$, where $\gamma_2$ is the second eigenvalue associated with their nonlinear hypergraph diffusion operator~\cite{chan2018spectral}. Unlike the normalized Laplacian considered here, their diffusion operator is nonlinear and depends on the state of the diffusion process, rather than being represented by a fixed matrix. In particular, the operator changes with the evolving distribution of mass across the vertices, and its eigenvalue is defined through this nonlinear diffusion framework. Consequently, $\gamma_2$ is not the second smallest eigenvalue of a fixed normalized Laplacian matrix, and the resulting bound is not directly comparable to the spectral bounds established here.  

{\thesisrev{Next, we improve Cheeger's inequality for non-covering hypergraphs. For any graph G that is not complete, we have $\evalNLH_2(G)\leq 1$~\cite{chung1996laplacians}. Xu and Zhou established the following analogous result for hypergraphs~\cite{xu2024normalized}. \ignore{For completeness, we reproduce the proof here for the convenience of the reader.}

\begin{theorem}
    Let $H$ be a non-covering hypergraph, then we have $\evalNLH_2(H)\leq 1$.
\end{theorem}
\ignore{\begin{proof}
    Since $H$ is non-covering, there exist two vertices $u$ and $v$ in $H$ for which no hyperedge contains both $u$ and $v$. Define $\bfr\in \mathbb{R}^{V(H)}$ by setting $\bfr(u)=d_H(v)$, $\bfr(v)=-d_H(u)$, and $\bfr(w)=0$ for every $w\notin\{u,v\}$. Together with Eq.~\ref{eq:lap-quad-form}, this implies that $\bfy^TL_H\bfy=d_H(u)d_H^2(v)+d_H^2(u)d_H(v)$. Moreover, we have $\bfy^TD_H\bfy=d_H(u)d_H^2(v)+d_H^2(u)d_H(v)$. Because the vector $\bfr$ is orthogonal to $\hat{\bfpsi}_1$, Eq.~\ref{eq:nu2} gives
    \[\evalNLH_2(H)\leq\frac{\bfr^TL_H\bfr}{\bfr^TD_H\bfr}=1\]
    which completes the proof.
\end{proof}}
}}
Using the spectral characterization established above, we derive a further improvement of the Cheeger's inequality that relates the second-smallest eigenvalue of the normalized Laplacian to the conductance of non-covering hypergraphs. Banerjee's bound specializes to $\phih(H) \leq (\kappa-1)\sqrt{(2-\evalNLH_2(H))\evalNLH_2(H)}$ as also established by Xu and Zhou for a $\kappa$-uniform hypergraph $H$~\cite{xu2024normalized}. We extend this result to non-covering hypergraphs and obtain a sharper bound.
\begin{theorem}[Improved Cheeger's inequality for non-covering hypergraphs]
\label{thm:cheeger-ineq-hypergraph}
    Let $H$ be a non-covering hypergraph. Then, we have
\begin{equation}
  \label{eq:cheeger-ineq-hypergraph}
  \frac{\evalNLH_2(H)}{2} \leq \newphi(H) \leq \sqrt{(2-\evalNLH_2(H))\evalNLH_2(H)}
\end{equation}
\end{theorem}
The proof of the lower bound follows from Eq.~(\ref{eq:cheeger-ineq-2-graph-diffusion}). Improving the upper bound requires more work, so we have moved it to Section~\ref{thm:cheeger-ineq-upper-bound-proof}.

\begin{corollary}
    \label{thm:cheeger-ineq-hypergraph-cor}
    {\thesisrev{Let $H$ be a non-covering hypergraph.}} Then, we have
\begin{equation}
  \label{eq:cheeger-ineq-hypergraph-cor}
  \frac{\evalNLH_2(H)}{\Upsilon_H} \leq \phih(H) \leq \sqrt{(2-\evalNLH_2(H))\evalNLH_2(H)}
\end{equation}
where $\Upsilon_H$ is rank of $H$.
\end{corollary}
\begin{proof}
    The proof follows from Theorem~\ref{thm:cheeger-ineq-hypergraph} and Eq.~\eqref{eq:banerjee}.
\end{proof}
{\thesisrev{ Note that \eqref{eq:cheeger-ineq-hypergraph} also holds for $2$-graphs which are not complete because only complete graphs have the second smallest eigenvalue of their normalized Laplacian greater than $1$~\cite{chung1997spectral}. In Section~\ref{section:cheeger-inequality:tightness:upper}, we further show that the cycle $2$-graph is tight for the upper bounds in both \eqref{eq:cheeger-graph} and \eqref{eq:cheeger-ineq-hypergraph}. This implies that \eqref{eq:cheeger-ineq-hypergraph} provides a strictly better upper bound on the conductance for non-complete $2$-graphs compared to \eqref{eq:cheeger-graph}.}}

In Section~\ref{sec:algorithmic}, we analyze Fiedler’s two-way global partitioning algorithm using our Cheeger’s inequality, which leverages the eigenvector associated with the second smallest eigenvalue of the normalized Laplacian.

\ignore{
The upper bound in~\eq
\begin{lemma}
\label{thm:cheeger-ineq-upper-bound}
Let $H$ be a connected hypergraph and let $\evalNLH_2(H)$ be the second smallest eigenvalue of its normalized Laplacian $\calL_H$. Then, we have
\begin{equation}
  \label{eq:cheeger-upper-bound}
  \newphi(H) \leq \sqrt{(2-\evalNLH_2(H))\evalNLH_2(H)}.
\end{equation}
\end{lemma}
We give proof of Lemma~\ref{thm:cheeger-ineq-upper-bound} in Section~\ref{thm:cheeger-ineq-upper-bound-proof}.
\begin{corollary}
\label{thm:cheeger-ineq-complete}
Let $H$ be a connected hypergraph and let $\evalNLH_2(H)$ be the second smallest eigenvalue of its normalized Laplacian $\calL_H$. Then, we have
\begin{equation}
  \label{eq:cheeger-complete}
  \frac{\evalNLH_2(H)}{2} \leq \newphi(H) \leq \sqrt{(2-\evalNLH_2(H))\evalNLH_2(H)}
\end{equation}
\end{corollary}
The proof of this corollary directly follows from Lemma~\ref{thm:cheeger-ineq} and Lemma~\ref{thm:cheeger-ineq-upper-bound} because $\sqrt{2\evalNLH_2(H)} \geq \sqrt{(2-\evalNLH_2(H))\evalNLH_2(H)}$.
}
\ignore{The lower bound is easy to check. For any $S \subseteq V$ if we define $\bm{y}=\bm{\chi}_S-\sigma\bm{1}$ where  $\bm{\chi}_S$ is the indicator vector of the set $S$ and $\sigma=\frac{\vol(S)}{\vol(V)}$, we can show that $\bm{y} \perp \bfpsi_1$ and so we can use $\bm{y}$ to deduce that $\hat{\phi}(S)$ is lower bounded by $\nu_2/2$. The detailed calculation can be found in Section~\ref{}. The proof of the upper bound is more involved and can be found in Section~\ref{}.}
\section{Construction of Optimal Hypergraph Expanders}
In this section, we construct hypergraph expanders from families of
$k$-regular $2$-graph expanders with a uniform two-sided spectral gap,
where $k\geq 3$ is fixed. This spectral assumption ensures that the
constructed hypergraphs have a uniform positive normalized
Laplacian gap and hence form expander families. In particular,
connected non-bipartite Ramanujan graphs satisfy this assumption.
Applying our construction to such graphs yields a
family of hypergraph expanders that is asymptotically tight for the
hypergraph Alon--Boppana bound. We also show that the constructed
families have a uniform positive spectral gap for the nonlinear
hypergraph diffusion, yielding logarithmic mixing-time upper
bounds for fixed accuracy. We first state the definition of
hypergraph expanders.

\label{section:expander-hypergraphs}
\begin{definition}[Expander hypergraph]
\label{def:expander-hypergraphs}
    Let $\scrH = \{H_i : i \in \scrI\}$ be a class of hypergraphs where $\scrI$ is an index set. The class $\scrH$ is called an {\em expander family} or a {\em family of expanders} if there exist constants $d>0$ and $\phi>0$ such that
    \begin{enumerate}
        \item Given a positive integer $\ell$, there exist at most finitely many hypergraphs in $\scrH$ having at most $\ell$ vertices.
        \item $\underset{v \in V_H}{\max}\, d_H(v) \leq d \text{ for all } H \in \scrH$.
        \item\label{item:expander} $\phih(H)\geq \phi \text{ for all } H \in \scrH$.
    \end{enumerate}
\end{definition}
The following Alon-Boppana-like bound is known for hypergraphs.
\begin{theorem}[Feng and Li~\cite{feng1996spectra}]\label{alon-boppanna-bound}
    Let $H$ be a $\kappa$-uniform and $d$-regular hypergraph. If there exist two hyperedges with at least $2t+2$ distance between them, then the second smallest eigenvalue of the Laplacian of $H$ satisfies
    \[\mu_2(H)\leq d-\frac{\kappa-2}{\kappa-1}-2\sqrt{\frac{d-1}{\kappa-1}}+\frac{1}{t+1}\left(2\sqrt{\frac{d-1}{\kappa-1}}-\frac{1}{\kappa-1}\right)\]
    where $\frac{1}{t+1}\left(2\sqrt{\frac{d-1}{\kappa-1}}-\frac{1}{\kappa-1}\right)$ tends to $0$ for all fixed values of $d$ and $\kappa$ as $n\longrightarrow\infty$.
\end{theorem}
Corollary~\ref{expander-corollary} presents hypergraph expander constructions that achieve this bound.

\begin{lemma}\label{thm:expander-1}
  Given an expander family $\scrG = \{G(k,i)\}_{i\in \calI}$ of non-bipartite Ramanujan $2$-graphs where $G(k,i)$ is a non-bipartite $k$-regular $2$-graph and $\calI$ is an index set. For all $r:0 \leq r \leq k-2$, there exists an expander family $\scrK_{\scrG}^{(r)}=\{K_{G(k,i)}^{(r)}\}_{i\in \calI}$ of $(k-r)$-uniform and regular hypergraphs with degree $\frac{k!}{(k-1-r)!}$ such that $\evalNLH_2\left(K_{G(k,i)}^{(r)}\right)\geq \frac{(k-2)^2}{k(k-1)}$ for all $i\in\calI$.
\end{lemma}
\begin{proof}
        Construct hypergraph $K_{G(k,i)}^{(0)}$ from $G(k,i)$ as follows. The vertex set of $K_{G(k,i)}^{(0)}$ is same as that of $G(k,i)$ i.e. $V_{K_{G(k,i)}^{(0)}}=V_{G(k,i)}$. The hyperedge set $E_{K_{G(k,i)}^{(0)}}$ of $K_{G(k,i)}^{(0)}$ is defined by $E_{K_{G(k,i)}^{(0)}}=\{e_v:v\in V_{G(k,i)}\}$ where hyperedge $e_v$ consists of all one-hop neighbors of $v$ in $G(k,i)$ ($e_v$ does not contain $v$). So, there are total $n$ hyperedges in $K_{G(k,i)}^{(0)}$. The hypergraph $K_{G(k,i)}^{(0)}$ is $k$-regular and $k$-uniform. If $A_{G(k,i)}$ is adjacency matrix of $G(k,i)$, then the adjacency matrix of $K_{G(k,i)}^{(0)}$ is $\frac{1}{k-1}B$ where $B=A_{G(k,i)}^2-kI$. It implies that normalized Laplacian of $K_{G(k,i)}^{(0)}$ is $I-\frac{1}{k(k-1)}B$.

    We construct hypergraph $K_{G(k,i)}^{(1)}$ from $K_{G(k,i)}^{(0)}$ as follows. The vertex set of $K_{G(k,i)}^{(1)}$ is same as that of $K_{G(k,i)}^{(0)}$. The hyperedge set $E_{K_{G(k,i)}^{(1)}}$ of $K_{G(k,i)}^{(1)}$ is defined as the multiset $\{e\setminus\{v\}:v\in e\text{ and }e\in E_{K_{G(k,i)}^{(0)}}\}$ with duplicate hyperedges retained when they arise from distinct hyperedges in $E_{K_{G(k,i)}^{(0)}}$. Any hyperedge of $K_{G(k,i)}^{(0)}$ is a set of $k$ vertices. We replace each hyperedge of $K_{G(k,i)}^{(0)}$ by all possible $k$ subsets (new hyperedges) of size $k-1$. The hypergraph $K_{G(k,i)}^{(1)}$ is $k(k-1)$-regular and $(k-1)$-uniform. Then, adjacency matrix of $K_{G(k,i)}^{(1)}$ is $B$. It implies that normalized Laplacian of $K_{G(k,i)}^{(1)}$ is $I-\frac{1}{k(k-1)}B$. 

    In general, we construct hypergraph $K_{G(k,i)}^{(r)}$ from $K_{G(k,i)}^{(r-1)}$  for $r:1\leq r\leq k-2$ as follows. The vertex set of $K_{G(k,i)}^{(r)}$ is same as that of $K_{G(k,i)}^{(r-1)}$. The hyperedge set $E_{K_{G(k,i)}^{(r)}}$ of $K_{G(k,i)}^{(r)}$ is defined as the multiset $\{e\setminus\{v\}:v\in e\text{ and }e\in E_{K_{G(k,i)}^{(r-1)}}\}$ with duplicate hyperedges retained when they arise from distinct hyperedges in $E_{K_{G(k,i)}^{(r-1)}}$. We replace each hyperedge of $K_{G(k,i)}^{(r-1)}$ by all possible $\binom{k+1-r}{k-r}$ subsets (new hyperedges) of size $k-r$. The hypergraph $K_{G(k,i)}^{(r)}$ is regular with degree $\frac{k!}{(k-1-r)!}$ and $(k-r)$-uniform. Then, adjacency matrix of $K_{G(k,i)}^{(r)}$ is $\frac{(k-2)!}{(k-1-r)!}B$. This implies that normalized Laplacian of $K_{G(k,i)}^{(r)}$ is $I-\frac{1}{k(k-1)}B$.

    If we use the non-bipartite $k$-regular Ramanujan graphs explicitly constructed by Lubotzky, Phillips and Sarnak~\cite{lubotsky1988ramanujan}; the bounds $|\evalAG_{\ell}(A_{G(k,i)}))|\leq 2\sqrt{k-1}$ for all $\ell:2\leq\ell\leq |V_{G(k,i)}|$ imply that $\evalNLH_2\left(K_{G(k,i)}^{(r)}\right)\geq \frac{k^2-4(k-1)}{k(k-1)}
=\frac{(k-2)^2}{k(k-1)}$ because the normalized Laplacian matrix of $K_{G(k,i)}^{(r)}$ is $\frac{1}{k(k-1)}\left(k^2I-A_{G(k,i)}^2\right)$.
\end{proof}

\begin{lemma}\label{thm:expander-2}
  Given an expander family $\scrG = \{G(k,i)\}_{i\in \calI}$ of non-bipartite Ramanujan $2$-graphs where $G(k,i)$ is a non-bipartite $k$-regular $2$-graph and $\calI$ is an index set. For all $r:0 \leq r \leq k-2$, there exists an expander family $\scrH_{\scrG}^{(r)}=\{H_{G(k,i)}^{(r)}\}_{i\in \calI}$ of $(k-r)$-uniform and regular hypergraphs with degree $\frac{k!}{(k-1-r)!r!}$ such that $\evalNLH_2\left(H_{G(k,i)}^{(r)}\right)\geq \frac{(k-2)^2}{k(k-1)}$ for all $i\in\calI$.
\end{lemma}
\begin{proof}
    For all $r:0 \leq r \leq k-2$, we construct the expander family $\scrH_{\scrG}^{(r)}=\{H_{G(k,i)}^{(r)}\}_{i\in \calI}$ of hypergraphs as follows. The vertex set of $H_{G(k,i)}^{(r)} \in \scrH_{\scrG}^{(r)}$ is the same as that of $G(k,i) \in \scrG$ i.e. $V_{H_{G(k,i)}^{(r)}}=V_{G(k,i)}$ for all $r:0 \leq r \leq k-2$ and for all $i\in\calI$. For $v\in V_{H_{G(k,i)}^{(r)}}$, let $N(v)$ be the one-hop neighborhood of $v$ in $V_{G(k,i)}$ ($N(v)$ does not contain $v$). Since $G(k,i)$ is $k$-regular, $|N(v)|=k$ for all $v\in G(k,i)$. The hyperedge set $E_{H_{G(k,i)}^{(r)}}$ of $H_{G(k,i)}^{(r)}$ is defined as the multiset $\{e:e\subseteq N(v),~|e|=k-r,\text{ and }  v\in V_{G(k,i)}\}$ with duplicate hyperedges retained when they arise from distinct neighborhoods. For every vertex $v\in V_{H_{G(k,i)}^{(r)}}$, the hyperedge set $E_{H_{G(k,i)}^{(r)}}$ consists of all possible $\comb[k]{k-r}$ subsets (each of size $k-r$) of $N(v)$ as hyperedges. Note that $H_{G(k,i)}^{(r)}$ is $(k-r)$-uniform and regular hypergraphs with degree $\frac{k!}{(k-1-r)!r!}$ for all $r:0 \leq r \leq k-2$.    

    Consider the expander family $\scrK_{\scrG}^{(r)}=\{K_{G(k,i)}^{(r)}\}_{i\in \calI}$ defined in Lemma \ref{thm:expander-1}. Note that $V_{H_{G(k,i)}^{(r)}}=V_{K_{G(k,i)}^{(r)}}=V_{G(k,i)}$ for all $r:0 \leq r \leq k-2$ and for all $i\in\calI$. The hyperedge set of $K_{G(k,i)}^{(r)}$ contains exactly $r!$ copies of each hyperedge of $H_{G(k,i)}^{(r)}$. In other words, the sets of hyperedges of $K_{G(k,i)}^{(r)}$ and $H_{G(k,i)}^{(r)}$ are identical except for the fact that each hyperedge $e\in E_{K_{G(k,i)}^{(r)}}$ has weight $r!w_{H_{G(k,i)}^{(r)}}(e)$ where $w_{H_{G(k,i)}^{(r)}}(e)$ is the weight of hyperedge $e\in H_{G(k,i)}^{(r)}$. Since multiplying all hyperedge weights by a constant does not change the normalized Laplacian matrix, so we have $\calL_{H_{G(k,i)}^{(r)}}=\calL_{K_{G(k,i)}^{(r)}}=\frac{1}{k(k-1)}\left(k^2I-A_{G(k,i)}^2\right)$ for all $r:0 \leq r \leq k-2$ and for all $i\in\calI$. Therefore, we have
    \[\evalNLH_2\left(H_{G(k,i)}^{(r)}\right)\geq \frac{(k-2)^2}{k(k-1)}\]
    for all $r:0 \leq r \leq k-2$ and for all $i\in\calI$.
\end{proof}

Cheeger's inequality~\eqref{eq:cheeger} allows us to express $\phih(H)\geq \phi$ equivalently in terms of the second-smallest eigenvalue of the normalized Laplacian $\text{ for all } H \in \scrH$. Furthermore, let $\scrG = \{G(k,i)\}_{i\in \calI}$ be a family of $k$-regular expanders with a uniform two-sided spectral gap, where $k$ is fixed. The family need not be Ramanujan. For each fixed $0\leq r\leq k-2$, the normalized Laplacian $\frac{1}{k(k-1)}\left(k^2I-A_{G(k,i)}^2\right)$ has a spectral gap bounded below by a positive constant independent of index $i$. Together with Cheeger's inequality~\eqref{eq:cheeger} and the bounded degrees of the constructed hypergraphs, this implies that $\scrK_{\scrG}^{(r)}$ and $\scrH_{\scrG}^{(r)}$ are hypergraph expander families. Corrolary~\ref{expander-corollary} further establishes that $\scrH_{\scrG}^{(0)}$ and $\scrK_{\scrG}^{(0)}$ are identical and form optimal expander families of hypergraphs when $\scrG = \{G(k,i)\}_{i\in \calI}$ is an expander family of non-bipartite Ramanujan graphs.

\ignore{The proofs of Lemmas~\ref{thm:expander-1}~\&~\ref{thm:expander-2}\ are given in Section~\ref{sec:proof:expanders}}

\begin{corollary}\label{expander-corollary}
    The expander families $\scrH_{\scrG}^{(0)}$ and $\scrK_{\scrG}^{(0)}$ are identical and asymptotically tight for the Alon-Boppana bound given in Theorem \ref{alon-boppanna-bound}.
\end{corollary}
\begin{proof}
    Lemmas \ref{thm:expander-1} and \ref{thm:expander-2} imply that $\evalNLH_2\left(H_{G(k,i)}^{(r)}\right)=\evalNLH_2\left(K_{G(k,i)}^{(r)}\right)\geq \frac{(k-2)^2}{k(k-1)}$ for all $r:0 \leq r \leq k-2$ and for all $i\in\calI$. Theorem \ref{alon-boppanna-bound} implies that
    \[\evalNLH_2\left(H_{G(k,i)}^{(0)}\right)=\evalNLH_2\left(K_{G(k,i)}^{(0)}\right)\leq 1-\frac{\kappa-2}{d(\kappa-1)}-\frac{2}{d}\sqrt{\frac{d-1}{\kappa-1}}=1-\frac{3k-4}{k(k-1)}=\frac{(k-2)^2}{k(k-1)}\]
    for all $i\in\calI$. This completes the proof.
\end{proof}

\paragraph{Why bipartite starting graphs cause a problem.}
The obstruction arises from the neighborhood construction itself.
Given a graph $G$, the construction forms hyperedges
\[
e_v=N_G(v),
\]
or nonempty subsets of these neighborhoods, where the center $v$ is
not included in $N_G(v)$.

Suppose that $G$ is bipartite with vertex classes $L$ and $R$. Then
\[
N_G(v)\subseteq
\begin{cases}
R, & v\in L,\\
L, & v\in R.
\end{cases}
\]
Consequently, every hyperedge of the constructed hypergraph $H$ lies
entirely within one of the two classes. This remains true when
hyperedges are replaced by smaller subsets. Hence, no hyperedge
crosses the cut $(L,R)$, and $H$ is disconnected.

In particular, if the initial mass is supported entirely on $L$,
the hypergraph diffusion cannot transfer any mass to $R$.
It therefore cannot converge to the uniform distribution on all
vertices. Moreover, the absence of hyperedges crossing $(L,R)$ gives
\[
\phi(H)=0
\qquad\text{and}\qquad
\nu_2(H)=0.
\]

Thus, although a bipartite graph may itself be an excellent expander,
the neighborhood construction produces a disconnected hypergraph
when applied to it.

More generally, for a family of $k$-regular starting graphs, a
sufficient condition is a uniform two-sided spectral gap: there
exists a constant $\delta>0$, independent of the member of the
family, such that
\[
|\lambda_i(A_G)|\leq (1-\delta)k,
\qquad 2\leq i\leq |V(G)|,
\]
where $\lambda_1(A_G)=k$ corresponds to the constant eigenvector.
This condition excludes both bipartite graphs, which have an
adjacency eigenvalue equal to $-k$, and families whose smallest
adjacency eigenvalue approaches $-k$. Merely requiring each starting
graph to be non-bipartite does not ensure this uniform spectral gap
and is therefore insufficient to guarantee that the constructed
hypergraphs form an expander family.

\paragraph{Why non-bipartite Ramanujan graphs meet the requirement.}
For a connected non-bipartite $k$-regular Ramanujan graph $G$, every
nontrivial adjacency eigenvalue satisfies~\cite{lubotsky1988ramanujan}
\[
|\lambda_i(A_G)|\leq 2\sqrt{k-1},
\qquad 2\leq i\leq |V(G)|.
\]
For fixed $k\geq 3$, we have
\[
2\sqrt{k-1}<k.
\]
Hence, the constant
\[
\delta=1-\frac{2\sqrt{k-1}}{k}>0
\]
provides a uniform two-sided spectral gap $|\lambda_i(A_G)|\leq (1-\delta)k,~2\leq i\leq |V(G)|$ throughout the family.
Consequently, the normalized Laplacian gap of each constructed
hypergraph $H$ satisfies
\[
\nu_2(H)\geq
\frac{k^2-4(k-1)}{k(k-1)}
=
\frac{(k-2)^2}{k(k-1)}>0.
\]
This lower bound is independent of the number of vertices.
Thus, non-bipartite Ramanujan starting graphs provide the required
uniform guarantee, whereas non-bipartiteness alone does not.

\section{Tightness results}
\label{section:cheeger-inequality:tightness}

We now show that Theorem~\ref{thm:cheeger-ineq-general} is tight on the lower side and Theorem~\ref{thm:cheeger-ineq-hypergraph} is tight on both sides if we consider both the rank and the number of nodes as parameters. As a byproduct, we will show that Eq.~(\ref{eq:banerjee}) is tight on both sides, i.e., that there are examples for which the diffusion conductance is precisely equal to the conductance. There are examples for which the diffusion conductance is $\Theta(\Upsilon_H)$ times the conductance.

\subsection{Tightness example for the lower bounds of Theorem~\ref{thm:cheeger-ineq-general} and Theorem~\ref{thm:cheeger-ineq-hypergraph}: The cube hypergraph}
\label{section:cheeger-inequality:tightness:lower}
In 2-graphs, the hypercube is known to be a tight example for the lower bound of Cheeger's inequality. Here, we find that a generalization called the cube hypergraph is tight for the lower bound of both Cheeger's inequalities for hypergraphs and non-covering hypergraphs. As a bonus, this hypergraph also proves the tightness of the upper bound of the relationship between the two versions of conductance, i.e., Eq.~\eqref{eq:banerjee}.

We now formally define the {\em $(\ell, k)$-cube hypergraph}, $H_{\ell,k}$, for $k \geq 2$ and $\ell \geq 1$. The $(1, k)$-cube hypergraph $H_{1,k}$ is a hypergraph with $k$ vertices and a single hyperedge containing all $k$ vertices.   For $\ell > 1$ we require the following notation: If $\bm{z} \in \bbR^{\eta}$ then for any $i \in [\eta]$, $\bm{z}^{(i)} \in \bbR^{\eta-1}$ is the $(\eta-1)$-dimensional vector obtained by omitting the $i$-th coordinate of $\bm{z}$. Now we define $H_{\ell,k} = (V,E)$ where $V = \bbZ_k^\ell$. For each $\bm{x} \in \bbZ_k^{\ell -1}$ and each $i \in [\ell]$, we define hyperedge  $e_{\bm{x},i} = \{\bm{u}\in\bbZ_k^\ell: \bm{u}^{(i)} = \bm{x}, \bm{u}(i)\in \bbZ_k\}$. And $E = \{e_{\bm{x},i}: \bm{x} \in \bbZ_k^{\ell -1},i \in [\ell]\}.$ Note that the number of vertices $n = |V| = k^\ell$.

In order to find an upper bound on the conductance of $H_{\ell,k}$ we consider $S \subseteq V$ defined by $$S = \{\bm{u}\in\bbZ_k^\ell: \bm{u}(1)\in \bbZ_{k/2}\},$$
assuming that $k$ is even for simplicity of presentation. There are a total of $k^{\ell-1}$ hyperedges $e_{\bm{x},1}$ that intersect with both $S$ and $V\setminus S$ where $\bm{x} \in \bbZ_{k}^{\ell -1}$. Any such hyperedge has an equal number of $\frac{k}{2}$ vertices in each partition. Therefore, we get that conductance is at most
\[\newphi(H_{\ell,k}) \leq \newphi(S)  = \frac{k^{l-1}\cdot\frac{1}{k-1}\cdot\frac{k^2}{4}}{\ell \cdot\frac{n}{2}}=\calO\left(\frac{1}{\ell}\right)=\calO\left(\frac{\log k}{\log n}\right).\]
This also implies that 
\[\phih(H_{\ell,k}) \leq \phih(S)  = \frac{k^{l-1}}{\ell \cdot\frac{n}{2}}=\calO\left(\frac{1}{k\ell}\right)=\calO\left(\frac{\log k}{k\log n}\right).\]
The results of Banerjee on the spectrum of the adjacency matrix can be used to determine $\evalNLH_2(H_{\ell,k})$~\cite{banerjee2021spectrum}. We explicitly calculate this quantity using a similar approach, based on the following observation. We first note that the adjacency matrix of $H_{1,k}$ coincides with that of the complete $2$-graph $K_k$ scaled by a factor of $1/(k-1)$. This can be extended as follows.
\begin{proposition}
\label{clm:clique-hypergraph}  
  For $\ell \geq 1$, let $K_k^{(i)}$, $i \in [\ell]$ be copies of the $k$-clique and let $K_{\ell,k} = \prod_{i \in [\ell]} K_k^{(i)}$ be the graph product of these copies. Then
  $$\evalNLH_2\left(H_{\ell,k}\right) = \evalNLG_2\left(K_{\ell,k}\right)=\Theta\left(\frac{1}{l}\right)=\Theta\left(\frac{\log k}{\log n}\right)$$
\end{proposition}
\begin{proof}
    The definition of $H_{\ell,k}$ given in Section~\ref{section:cheeger-inequality:tightness:lower} implies that $H_{1,k}$ consists of a single hyperedge containing $k$ vertices $0,~1,~\ldots,~i,~\ldots,~k-1$. The hypergraph $H_{2,k}$ is obtained from $H_{1,k}$ as follows. The hypergraph $H_{2,k}$ consists of $k$ copies of $H_{1,k}$. Let these $k$ copies of $H_{1,k}$ be denoted by $H_{1,k}^{(1)},~H_{1,k}^{(2)},~\ldots,~H_{1,k}^{(k)}$. Insert $k$ new hyperedges $e_1,~e_2,~\ldots,~e_i,~\ldots,~e_k$ as follows. The hyperedge $e_i$ consists of $k$ copies of vertex $i$ from $H_{1,k}^{(1)},~H_{1,k}^{(2)},~\ldots,~H_{1,k}^{(k)}$. Similarly, $H_{3,k}$ is obtained from $k$ copies $H_{2,k}^{(1)},~H_{2,k}^{(2)},~\ldots,~H_{2,k}^{(k)}$ of $H_{2,k}$ by inserting new $k^2$ hyperedges. In general, $H_{\ell,k}$ is obtained from $k$ copies $H_{\ell-1,k}^{(1)},~H_{\ell-1,k}^{(2)},~\ldots,~H_{\ell-1,k}^{(k)}$ of $H_{\ell-1,k}$ by inserting new $k^{l-1}$ hyperedges. 

    Let $K_k$ denote the complete $2$-graph with $k$ vertices. For $\ell \geq 1$ and $i \in [\ell]$, let $K_k^{(i)}$ be $\ell$ copies of $K_k$. Define $K_{\ell,k} = \prod_{i \in [\ell]} K_k^{(i)}=K_{\ell-1,k}\times K_k$ for all $l\geq 2$. The Laplacian of $K_k$ has eigenvalue $0$ with multiplicity one and eigenvalue $k$ with multiplicity $k-1$. It implies that $k$ is the second smallest eigenvalue of the Laplacian of $K_{\ell,k}$ for all positive integers $l$~\cite{merris:1994}. The degree of each vertex of $K_{\ell,k}$ is $l(k-1)$. Therefore, second smallest eigenvalue of normalized Laplacian of $K_{\ell,k}$ is $\Theta\left(\frac{1}{l}\right)=\Theta\left(\frac{\log k}{\log n}\right)$.

    It is important to note that adjacency matrix of $H_{\ell,k}$ is $\frac{1}{k-1}$ times that of $K_{\ell,k}$ for all positive integers $l$. Therefore, the normalized Laplacian of $H_{\ell,k}$ is the same as that of $K_{\ell,k}$ for all positive integers $l$. It implies that second smallest eigenvalue of normalized Laplacian of $H_{\ell,k}$ is $\evalNLH_2(H_{\ell,k})=\evalNLG_2(K_{\ell,k})=\Theta\left(\frac{1}{l}\right)=\Theta\left(\frac{\log k}{\log n}\right)$.
\end{proof}

This proves that $H_{\ell,k}$ is a tight example for the lower bounds of Theorem~\ref{thm:cheeger-ineq-general} and Theorem~\ref{thm:cheeger-ineq-hypergraph}. Moreover, $\newphi(H_{\ell,k}) = k\phih(H_{\ell,k}) = \calO\left(\frac{\log k}{\log n}\right)$. Therefore, the hypergraph $H_{\ell,k}$ is also a tight example for the upper bound of Eq.~\eqref{eq:banerjee}.

\subsection{Tightness example for the upper bound of the diffusion conductance Cheeger's inequality Theorem~\ref{thm:cheeger-ineq-hypergraph}: The cycle hypergraph}
\label{section:cheeger-inequality:tightness:upper}
Unlike the Cheeger's inequality for conductance (Theorem~\ref{thm:cheeger-ineq-general}), the Cheeger's inequality for diffusion conductance (Theorem~\ref{thm:cheeger-ineq-hypergraph}) is tight on the upper side even if we consider the rank as a parameter. In $2$-graphs, the cycle provides tightness on the upper side of Cheeger's inequality, and here, a natural generalization of the cycle plays the same role.

Consider the following hypergraph that we call the {\em $(n,k)$-cycle hypergraph}, $C_{n,k}$, defined for $n \geq 3$ and $k:2\leq k<\frac{n}{2}$: $V = \bbZ_n$, $E = \{ \{i, i+1 \mod n,\ldots,i+k-1 \mod n\} : 0 \leq i < n\}$. When $k =2$, this is simply the 2-graph known as the $n$-cycle. We show in Proposition \ref{thm:cycle-conductance} that for any $i \in \bbZ_n$, the set $S_i = \{i, (i+1)\mod n, \ldots, (i + n/2-1)\mod n\}$ satisfies $S_i=\arg\min_{S:\emptyset\subsetneq S\subsetneq V} \newphi\left(S\right)$ where we assume for simplicity of the presentation that $n$ is even.

\begin{proposition}\label{thm:cycle-conductance}
    \begin{equation}\label{eq:cycle-conductance}
        \newphi(C_{n,k})=\frac{4\sum_{i=1}^{k-1}i(k-i)}{nk(k-1)} = \Theta\left(\frac{k^3}{nk^2}\right) = \Theta\left(\frac{k}{n}\right)
    \end{equation}
\end{proposition}
\begin{proof}
    If we take $S_j$ as the set of any $j$ consecutive vertices on the cycle, then $\min(\vol(S_j),\vol(V\setminus S_j))\leq\min(\vol(S_{\frac{n}{2}}),\vol(V\setminus S_{\frac{n}{2}}))$ and $\sum_{e:e\in \partialh(S_j)}\frac{1}{|e|-1}|e\cap S_j||e\cap (V\setminus S_j)|$ does not decrease proportionally to the same extent for $j\neq \frac{n}{2}$. In addition, if we take $T_j$ as the set of $j$ non-consecutive vertices, then $\sum_{e:e\in \partialh(S_j)}\frac{1}{|e|-1}|e\cap S_j||e\cap (V\setminus S_j)|\leq\sum_{e:e\in \partialh(T_j)}\frac{1}{|e|-1}|e\cap T_j||e\cap (V\setminus T_j)|$. 

    It implies that $\newphi(C_{n,k})$ is given by the following partition of the vertices of $C_{n,k}$. We divide vertex set of $C_{n,k}$ into two parts such that $(i+1)\mod n, (i+2)\mod n, \ldots, (i + n/2)\mod n$ are in one part and $(i+n/2+1)\mod n, (i+n/2+2)\mod n, \ldots, (i + n)\mod n$ are in the other part where $i\in\bbZ_n$. Then, there are a total of $2(k-1)$ hyperedges across this partition. Therefore, conductance $\newphi(C_{n,k})$ is
    \[\frac{4\sum_{i=1}^{k-1}i(k-i)}{nk(k-1)}\]
\end{proof}

We also have
\begin{equation}
  \label{eq:cycle-conductance-factor-k}
  \phih(C_{n,k}) = \frac{4(k-1)}{nk} = \Theta\left(\frac{1}{n}\right).
  \end{equation}

In order to compute the eigenvalues of the $\calL_{C_{n,k}}$ we note first the adjacency matrix of $C_{n,k}$ is same as that of the weighted additive Cayley $2$-graph $G=(\bbZ_n,S,\omega)$ where $S=\{\pm 1,~\pm 2,~\ldots,~\pm (k-1)\}$. We write $-1$ rather than $n-1$ to simplify the definition of the mapping $\omega$. The mapping $\omega$ assigns weight$\frac{k-|s|}{k-1}$ to $s\in S$ and weight $0$ to the elements in $V\setminus S$ such that $\omega(s)=\omega(-s)$ for all $s\in S$. Note that $\rchi_r(x)=\exp(\frac{2\pi xri}{n})$ is a character of $G=(\bbZ_n,S,\omega)$ for all $r=0,~1,~\ldots,~(n-1)$. Therefore, the second-largest eigenvalue of the adjacency matrix is
\begin{align*}
    \sum_{\underset{\hspace{-0.75cm}j\neq 0}{j=-(k-1)}}^{k-1}\frac{k-|j|}{k-1}\rchi_1(j)&=\frac{2}{k-1}\sum_{j=1}^{k-1}(k-j)\cos{\frac{2\pi j}{n}}\\
    &= \frac{2}{k-1}\sum_{j=1}^{k-1}(k-j)\left(1-\Theta\left(\frac{j^2}{n^2}\right)\right)\\
    &=k-\Theta\left(\frac{1}{n^2}\right)\frac{2}{k-1}\sum_{j=1}^{k-1}(k-j)j^2
\end{align*}
It implies that $\evalNLH_2(C_{n,k})=\Theta\left(\frac{1}{n^2}\right)\frac{2}{k(k-1)}\sum_{j=1}^{k-1}(k-j)j^2=\Theta\left(\frac{k^2}{n^2}\right)$.

\ignore{In order to compute the eigenvalues of the $\calL_{C_{n,k}}$ we note first that its adjacency matrix is the same as that of the weighted 2-graph $G = (V,E,w)$ with $V = \bbZ_n$, $E = \{\{i, j\} \in \bbZ_n^2: 1\leq |i - j| \leq k-1\}$, and for each $\{i,j\} \in E$, $w(\{i,j\}) = (k - |i - j|)/(k-1)$. The $2$-graph $G$ can be viewed as a weighted Cayley graph on the additive group $\bbZ_n$ with generators $S=\{\pm 1,~\pm 2,~\ldots,~\pm (k-1)\}$ and weight function $w(s) = \frac{k-|s|}{k-1}$ for each $s \in S$. Since $\rchi_r(x)=\exp(\frac{2\pi xri}{n})$ are the characters of the group for $0 \leq r \leq k-1$ and the characters are the eigenvectors of the adjacency matrix (and the Laplacian), we use $\rchi_1$ to compute
\ignore{\[ \evalNLH_2(C_{n,k}) \leq \frac{1}{k(k-1)}\left(\frac{2\pi}{n}\right)^2\sum_{j=1}^{k-1}(k-j)j^2 = \calO\left(\frac{k^4}{n^2k^2}\right) = \calO\left(\frac{k^2}{n^2}\right). \]}
\[ \evalNLH_2(C_{n,k}) = \Theta\left(\frac{k^2}{n^2}\right). \]}
Putting this together with~\eqref{eq:cycle-conductance}, we see that $C_{n,k}$ is a tight example for the upper bound of the diffusion conductance Cheeger inequality Theorem~\ref{thm:cheeger-ineq-hypergraph}. \ignore{The details of the eigenvalue calculation are given in Section~\ref{section:cheeger-inequality:tightness:upper-proof}.} This also proves that $C_{n,k}$ is a tight example for the upper bounds of Corollary~\ref{thm:cheeger-ineq-hypergraph-cor} up to a factor of $k$.

\subsection{\texorpdfstring{Tightness example for the lower bound of Eq.~\eqref{eq:banerjee}: The bicycle hypergraph}{Tightness example for the lower bound of Eq. (Banerjee): The bicycle hypergraph}}
\label{section:cheeger-inequality-tightness-upper-bicycle}
\ignore{We now present an example that is tight up to a factor of $\sqrt{r_H}$ for the upper bound of Theorem~\ref{thm:cheeger-ineq-hypergraph}. }This example is one in which the conductance and diffusion conductance are equal, i.e., a tight example for the lower bound in Eq.~\eqref{eq:banerjee}.

\begin{figure}[htbp]
\includegraphics[width=\columnwidth]{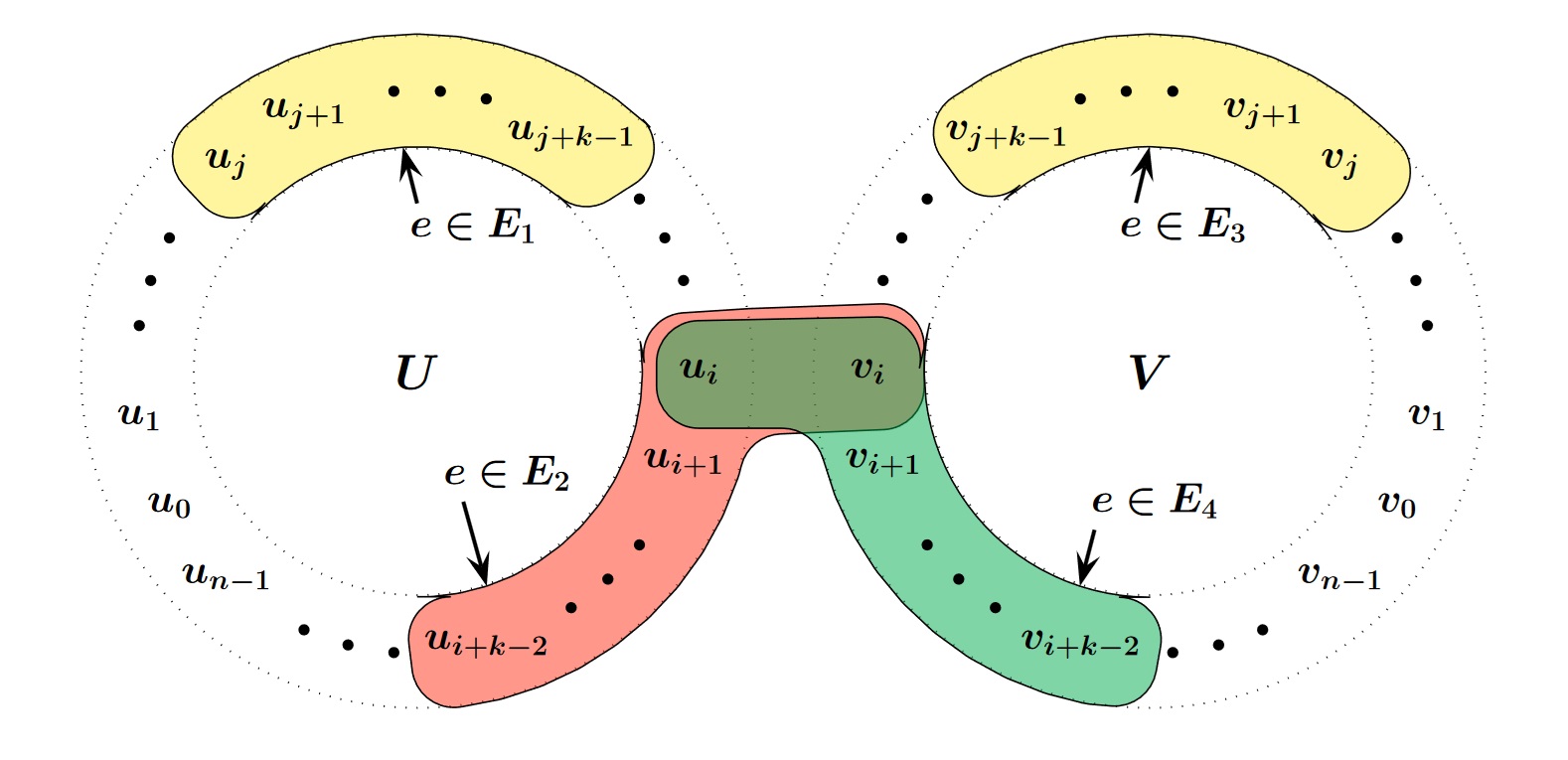}
\caption{The $(2n,k)$-bicycle hypergraph $C_{2n,k}$: only a few hyperedges are shown for illustration.}
\Description{The $(2n,k)$-bicycle hypergraph $C_{2n,k}$: only a few hyperedges are shown for illustration.}
\label{fig:hyperbicycle}
\end{figure}

Consider the following hypergraph that we call the {\em $(2n,k)$-bicycle hypergraph}, $C_{2n,k}$, defined for $n \geq 6$ and $k:3\leq k<\frac{n}{2}$. The vertex set of $C_{2n,k}$ is $V_{C_{2n,k}} = U\cup V$ where $U=\{u_0,~u_1,~\ldots,~u_{n-1}\}$ and $V=\{v_0,~v_1,~\ldots,~v_{n-1}\}$. The hyperedge set of $C_{2n,k}$ is $E_{C_{2n,k}}$ consisting of 4 disjoint sets: = $E_1, E_2, E_3$ and $E_4$. The sets $E_1$ and $E_3$ contain hyperedges $\{u_i, u_{i+1 \mod n}, \ldots, u_{i+k-1 \mod n}\}$  and $\{v_i, v_{i+1 \mod n}, \ldots, v_{i+k-1 \mod n}\}$ respectively for all $0 \leq i < n$. These we call {\em internal} hyperedges. The $i$-th hyperedge of $E_2$ contains a single vertex from $V$, $v_i$, and $k-1$ vertices from $U$, $u_i, u_{i+1 \mod n}, \ldots u_{i+k-2 \mod n}$, and $E_2$ contains such an edge for each $0 \leq i < n$. The hyperedges of $E_4$ symmetrically contain one vertex from $U$ and $k-1$ vertices from $V$.  We call $E_2$ and $E_4$ {\em cross hyperedges}. $C_{2n,k}$ is clearly a $k$-uniform hypergraph. See Figure~\ref{fig:hyperbicycle} for an illustration.

\ignore{
The sets $E_1,~E_2,~E_3,\text{ and }E_4$ of hyperedges are defined as follows
\begin{align*}
    E_1 &= \{ e_i : 0 \leq i \leq n-1\} \text{ where } e_i=\{u_j : j=i+\ell \mod n,~0 \leq \ell \leq k-1\}. \\
    E_2 &= \{ e_i : 0 \leq i \leq n-1\} \text{ where } e_i=\{v_i\}\cup\{u_j : j=i+\ell \mod n,~0 \leq \ell \leq k-2\}. \\
    E_3 &= \{ e_i : 0 \leq i \leq n-1\} \text{ where } e_i=\{v_j : j=i+\ell \mod n,~0 \leq \ell \leq k-1\}. \\
    E_4 &= \{ e_i : 0 \leq i \leq n-1\} \text{ where } e_i=\{u_i\}\cup\{v_j : j=i+\ell \mod n,~0 \leq \ell \leq k-2\}.
\end{align*}}
We assume that each cross hyperedge $e\in E_2\cup E_4$ has weight $w_1$ and that each internal hyperedge $e\in E_1\cup E_3$ has weight $w_2$, where $w_1=1/n$ and $w_2=1$. So each vertex $v\in V_{C_{2n,k}}$ has degree $k(w_1+w_2)$ i.e. $C_{2n,k}$ is $k(w_1+w_2)$-regular. It implies that $\vol_{C_{2n,k}}(U)=\vol_{C_{2n,k}}(V)=nk(w_1+w_2)$. We show in Proposition~\ref{thm:bicycle-conductance} that $U=\arg\min_{S:\emptyset\subsetneq S\subsetneq V_{C_{2n,k}}} \newphi\left(S\right)$. It is to be noted that the boundary of $U$,  $\partialh(U)$, is precisely the cross edges $E_2\cup E_4$. For each $e\in\partialh(U)$, either $|e\cap U|=1$ or $|e\cap U|=k-1$. Using these facts, we compute the conductance to be
\begin{equation}
  \label{eq:bicycle-newphi}
  \newphi(C_{2n,k}) = \frac{2\sum_{i=1}^{n}\frac{w_1}{k-1}(k-1)}{nk(w_1+w_2)} = \Theta\left(\frac{w_1}{k(w_1+w_2)}\right)=\Theta\left(\frac{1}{kn}\right).
  \end{equation}
Also, $|\partialh(U)|=2n$ implies that
\begin{equation}
  \label{eq:bicycle-phih}
  \phih(C_{2n,k}) = \frac{2nw_1}{nk(w_1+w_2)} = \Theta\left(\frac{w_1}{k(w_1+w_2)}\right)=\Theta\left(\frac{1}{kn}\right).
  \end{equation}
\ignore{In order to compute an upper bound on the second smallest eigenvalue of the normalized Laplacian $\calL_{C_{2n,k}}$, define $\sigma=\frac{\vol_{C_{2n,k}}(U)}{\Delta(C_{2n,k})}$. Also, define $2n$-vector $\bfx$ by $\bfx(v)=1-\sigma$ if $v\in U$ and $\bfx(v)=-\sigma$ if $v\in V$. It implies that $\bfx\perp \hat{\bfpsi}_1,~\bfx^TL_{C_{2n,k}}\bfx=2nw_1$, and $\bfx^TD_{C_{2n,k}}\bfx=\frac{\vol_{C_{2n,k}}(U)\vol_{C_{2n,k}}(V)}{\vol_{C_{2n,k}}(U\cup V)}$. Equation \eqref{eq:nu2} implies that
\[ \evalNLH_2(C_{2n,k}) \leq \frac{\bfx^TL_{C_{2n,k}}\bfx}{\bfx^TD_{C_{2n,k}}\bfx} = \frac{2nw_1\vol_{C_{2n,k}}(U\cup V)}{\vol_{C_{2n,k}}(U)\vol_{C_{2n,k}}(V)} = \frac{4n^2w_1k(w_1+w_2)}{n^2k^2(w_1+w_2)^2} = \calO\left(\frac{w_1}{k(w_1+w_2)}\right)=\Theta\left(\frac{1}{kn}\right). \]
This proves that $C_{2n,k}$ is a tight example for the upper bounds of Theorem~\ref{thm:cheeger-ineq-hypergraph} up to a factor of $\sqrt{k}$. }Equations \eqref{eq:bicycle-newphi} and \eqref{eq:bicycle-phih} implies that $C_{2n,k}$ is a tight example for the lower bound of Eq.~\eqref{eq:banerjee}.

\begin{proposition}\label{thm:bicycle-conductance}
    $\newphi(C_{2n,k})=\Theta\left(\frac{1}{kn}\right).$
\end{proposition}
\begin{proof}
        Let $S_j$ be the set of any $j$ consecutive vertices on the cycle $U$. Then, $\min(\vol_{C_{2n,k}}(U),\vol_{C_{2n,k}}(V))\geq\min(\vol_{C_{2n,k}}(S_j),\vol_{C_{2n,k}}(V_{C_{2n,k}}\setminus S_j))$ and $\sum_{e:e\in \partialh(S_j)}\frac{1}{|e|-1}|e\cap S_j||e\cap (V\setminus S_j)|$ does not decrease proportionally to the same extent for $S_j\subsetneq U$. In addition, if we take $T_j$ as the set of $j$ non-consecutive vertices, then $\sum_{e:e\in \partialh(S_j)}\frac{1}{|e|-1}|e\cap S_j||e\cap (V_{C_{2n,k}}\setminus S_j)|\leq\sum_{e:e\in \partialh(T_j)}\frac{1}{|e|-1}|e\cap T_j||e\cap (V_{C_{2n,k}}\setminus T_j)|$. It implies that $\newphi(S_j)\geq\newphi(U)$. Similarly, it follows that $\phih(S_j)\geq\phih(U)$. Similar analysis also works for the cycle $V$.

        Let $S_{i,\,j}$ be the subset of vertices containing $i$ consecutive vertices from cycle $U$ and $j$ consecutive vertices from cycle $V$. Similar to the proof of Proposition~\ref{thm:cycle-conductance}, we can show that $S_{n/2,\,n/2}=\underset{i,\,j}{\arg\min}\,\newphi(S_{i,\,j})$ and $\newphi(S_{n/2,\,n/2})=\Theta\left(\frac{k}{n}\right)$. Similarly, $S_{n/2,\,n/2}=\underset{i,\,j}{\arg\min}\,\phih(S_{i,\,j})$ and $\phih(S_{n/2,\,n/2})=\Theta\left(\frac{1}{n}\right)$.

        It implies that $U=\arg\min_{S:\emptyset\subsetneq S\subsetneq V_{C_{2n,k}}} \newphi\left(S\right)=\arg\min_{S:\emptyset\subsetneq S\subsetneq V_{C_{2n,k}}} \phih\left(S\right)$ and $\newphi(U)=\phih(U)=\Theta\left(\frac{1}{kn}\right)$.
\end{proof}

\section{Higher Order Cheeger Inequalities for Hypergraphs}\label{section:cheeger-inequality-high-order}
Our AE graph correspondence transfers the Lee–Gharan–Trevisan higher-order Cheeger inequalities to hypergraphs.
\begin{theorem}[Higher Order Cheeger's inequality for hypergraphs]
\label{thm:cheeger-ineq-hypergraph-high-order}
Let $\ell:1 < \ell < n$ be an integer and let $H=(V_H,E_H,w_H)$ be a hypergraph. If $\evalNLH_{\ell}(G)$ is the $\ell$-th smallest eigenvalue of its normalized Laplacian $\calL_H$. Then, we have
\begin{equation}
  \label{eq:cheeger-hypergraph-high-order}
  \frac{\evalNLH_{\ell}(H)}{\Upsilon_H} \leq \rhoh_{\ell}(H) \leq \calO({\ell}^2)\sqrt{\evalNLH_{\ell}(H)}
\end{equation}
\end{theorem}
\begin{proof}
    The proof easily follows from Lemma~\ref{lemma:high-order-cheeger-ineq} below and Eq.~\eqref{eq:trevisan}. 
\end{proof}
\begin{lemma}[Higher Order Cheeger's inequality for hypergraphs: Diffusion conductance version]
\label{lemma:high-order-cheeger-ineq}
Let $\ell:1 < \ell < n$ be an integer and let $H=(V_H,E_H,w_H)$ be a hypergraph. If $\evalNLH_{\ell}(G)$ is the $\ell$-th smallest eigenvalue of its normalized Laplacian $\calL_H$. Then, we have
\begin{equation}
  \label{eq:cheeger-high-order}
  \frac{\evalNLH_{\ell}(H)}{2} \leq \newrho_{\ell}(H) \leq \calO({\ell}^2)\sqrt{\evalNLH_{\ell}(H)}
\end{equation}
\end{lemma}
\begin{proof}
The higher order Cheeger's Inequality Theorem~\ref{thm:cheeger-ineq-graph-higher} for $2$-graphs implies that
\begin{equation}
  \label{eq:cheeger-ineq-2-hypergraph-high-order}
  \frac{\evalNLH_{\ell}(H)}{2} \leq \rhog_{\ell}(G_H) \leq \calO({\ell}^2)\sqrt{\evalNLH_{\ell}(H)},
\end{equation}
where $G_H$ is the AE graph of $H$.
Lemma~\ref{lemma:high-order-conductance-equivalence} implies that $\rhog_{\ell}(G_H)=\newrho_{\ell}(H)$ and so the result follows.
\end{proof}
\ignore{\paragraph{Tight example for the lower bound} We proved in Section~\ref{section:cheeger-inequality:tightness:lower} that $k$-th smallest eigenvalue of normalized Laplacian of hyperhypercube $K_{\ell,k}$ is $\evalNLH_k\left(K_{\ell,k}\right) = \Theta(1/\ell)$. For $i:0\leq i<k$, define $S_i = \{\bm{u}\in\bbZ_k^\ell: \bm{u}^{(1)} \in \bbZ_k^{\ell-1}, \bm{u}(1)=i\}$
\rkcomment{Add a tight example for lower bound: hyperhypercube.}}
\ignore{\paragraph{Comparison with higher-order Cheeger inequality proved by Mulas~\cite{mulas2021cheeger}}} 

\section{Algorithmic Implications of Our Results}
\label{sec:algorithmic}
\subsection{A global partitioning algorithm for hypergraphs}\label{sec:algorithmic:fiedler-algo}
\paragraph{The algorithm} Spielman and Teng~\cite{spielman1996spectral} and Mihail~\cite{mihail1989conductance} devised a global two-way partitioning algorithm for $2$-graphs based on an eigenvector corresponding to the second smallest eigenvalue of the normalized Laplacian. The same algorithm can be applied to the second eigenvector of the hypergraph's normalized Laplacian. We recall that the algorithm first sorts the eigenvector in non-increasing order. Then, it returns the subset of vertices with minimum conductance based on the sweep cuts of the sorted eigenvector, where a sweep cut is defined as a prefix of the sorted eigenvector.

\paragraph{Quality of the solution set} Let $S$ be the final cluster returned by the algorithm. The fact that $S$ is a minimum conductance sweep cut of the eigenvector $\bfx$ corresponding to the second smallest eigenvalue implies that $\phih(S)\leq\sqrt{2\frac{\bfx^TL_H\bfx}{\bfx^TD_H\bfx}}=\sqrt{2\evalNLH_2(H)}$ (c.f., e.g., Chung~\cite{chung1996laplacians}) and Theorem~\ref{thm:cheeger-ineq-hypergraph} implies that $\evalNLH_2(H)\leq 2\newphi(H)$. Therefore, we have $\phih(S)\leq 2\sqrt{\newphi(H)}$. It implies that $S$ satisfies $\phih(H) \leq \phih(S) \leq 2\sqrt{\newphi(H)}$. So, we get a bound for the conductance of the partition computed in terms of the diffusion conductance.

The algorithm's running time for hypergraphs is the same as in the case of $2$-graphs.

\ignore{
\paragraph{Running time} Sorting the eigenvector $\bfx$ corresponding to second smallest eigenvalue of normalized Laplacian of $H$ takes time $\calO(\lvert V_H\rvert\ln{\lvert V_H\rvert})$. Also, finding the sweep cut with the best conductance takes time $\calO(\Delta(H))$. This is the case because the computation of the conductance of a sweep cut from that of the sweep cut at the previous step requires all the hyperedges incident on the newly added vertex to the updated sweep cut to be processed for finding the edge boundary of the updated sweep cut. Therefore, running time of Fiedler's algorithm is $\calO(\lvert V_H\rvert\ln{\lvert V_H\rvert}+\Delta(H))$
}

\subsection{A multi-way partitioning algorithm for hypergraphs}
\label{sec:algorithmic:multi-way-clustering}
\paragraph{The algorithm} Lee, Gharan, and Trevisan~\cite{lee2014multiway} devised a multi-way partitioning algorithm for $2$-graphs which can easily be extended to hypergraphs as follows:
\begin{itemize}
\item Given the $\ell$ eigenvectors $\rh_1,~\rh_2,~\ldots,~\rh_{\ell}$ corresponding to the smallest $\ell$ normalized Laplacian eigenvalues define the mapping $F:V_H \longrightarrow \bbR^{\ell}$ by $F(v)\coloneqq (\rh_1(v),~\rh_2(v),~\ldots,~\rh_{\ell}(v))$ and the radial projection distance $d_F(u,v) \coloneqq \left\lVert \frac{F(u)}{\lVert F(u) \rVert} - \frac{F(v)}{\lVert F(v) \rVert} \right\rVert$.
\item Partition the vertex set $V_H$ into $\ell$ disjoint sets $S_1,~S_2,~\ldots,~S_{\ell}$ with the radial projection distance $d_F$ using the $\ell$-means algorithm.
\item  For all $i : 1 \leq i \leq \ell$, assume that the set $S_i$ is sorted in monotonically decreasing order of $\lVert F(\cdot) \rVert^2$ values. Return the minimum conductance sweep cut $T_i$ of $S_i$ for all $i : 1 \leq i \leq \ell$.
\end{itemize}  

\ignore{
\paragraph{Running time} The algorithm $\ell$-means takes time $\calO(n^{\ell^2+1})$. Sorting takes time $\calO(\sum_{i=1}^{\ell}\lvert S_i\rvert\ln{\lvert S_i\rvert})$, and finding the sweep cuts with best conductance takes time $\calO(\sum_{i=1}^{\ell}\vol_H(S_i))$. Therefore, the running time of the algorithm is $\calO(n^{\ell^2+1} + \sum_{i=1}^{\ell}\lvert S_i\rvert\ln{\lvert S_i\rvert} + \Delta(H))$. \rkcomment{Running time of $\ell$-means algorithm to be updated.}
}

\paragraph{Quality of the solution set} For all $i : 1 \leq i \leq \ell$, the set $T_i$ satisfies $\phih(T_i) \leq \calO(\ell^2)\sqrt{\evalNLH_{\ell}(H)}$ which implies that $\rhoh_{\ell}(T_1,~T_2,~\ldots,~T_{\ell}) \leq \calO(\ell^2)\sqrt{\evalNLH_{\ell}(H)}$ (c.f., e.g., Lee, Gharan, and Trevisan~\cite{lee2014multiway}). Also, Eq.~\eqref{eq:cheeger-high-order} implies that $\evalNLH_{\ell}(H) \leq 2\newrho_{\ell}(H)$. Therefore, we have $\rhoh_{\ell}(T_1,~T_2,~\ldots,~T_{\ell}) \leq \calO(\ell^2)\sqrt{\newrho_{\ell}(H)}$. This implies that the solution sets $T_1,~T_2,~\ldots,~T_{\ell}$ satisfy $\rhoh_{\ell}(H) \leq \rhoh_{\ell}(T_1,~T_2,~\ldots,~T_{\ell}) \leq \calO(\ell^2)\sqrt{\newrho_{\ell}(H)}$.

The algorithm's running time for hypergraphs is the same as in the case of $2$-graphs.

\section{Proof of upper bound of Theorem~\ref{thm:cheeger-ineq-hypergraph}} \label{thm:cheeger-ineq-upper-bound-proof}
Let $\rh_2$ be the eigenvector corresponding to the eigenvalue $\nu_2(H)$. Eq.~\eqref{eq:nu2} implies that
\begin{equation}\label{eq:eq1-cheeger-ineq-upper-bound-proof}
    \evalNLH_2(H)= \frac{\rh_2^T\calL_H\rh_2}{\rh_2^T\rh_2}=\frac{\bfx^TL_H\bfx}{\bfx^TD_H\bfx}
\end{equation}
\ignore{\abcomment{don't repeat this, it has already been defined earlier}}
where $\bfx=D_H^{-1/2}\rh_2$. Also, $\evalNLH_2(H)\rh_2=\calL_H\rh_2$ implies that $\evalNLH_2(H)D_H^{1/2}\bfx=\calL_HD_H^{1/2}\bfx=D_H^{-1/2}L_H\bfx$. For any vertex $u\in V_H$, we get the following by comparing the $u$-th element on both sides of the relation $\evalNLH_2(H)D_H^{1/2}\bfx=D_H^{-1/2}L_H\bfx$.
\begin{equation}\label{eq:eq2-cheeger-ineq-upper-bound-proof}
    \evalNLH_2(H)\sqrt{d_H(u)}\bfx(u) = \sqrt{d_H(u)}\bfx(u) - \sum_{v:v\in V_H\setminus\{u\}}\frac{A_H[u,v]}{\sqrt{d_H(u)}}\bfx(v)
\end{equation}
Multiplying Eq.~\eqref{eq:eq2-cheeger-ineq-upper-bound-proof} by $\sqrt{d_H(u)}\bfx(u)$, we get
\begin{equation}\label{eq:eq3-cheeger-ineq-upper-bound-proof}
    \evalNLH_2(H)d_H(u)(\bfx(u))^2 = \bfx(u)\left(d_H(u)\bfx(u) - \sum_{v:v\in V_H\setminus\{u\}}A_H[u,v]\bfx(v)\right)
\end{equation}
Let $S^+=\{u\in V_H : \bfx(u) > 0\}$. Define vector $\bfy$ by $\bfy(u)=\bfx(u)$ if $u\in S^+$, and 0 otherwise. Eq.~\eqref{eq:eq3-cheeger-ineq-upper-bound-proof} implies that
\ignore{Let $\{\bfx(u) : u \in V_H\}=\{\tau_1,~\tau_2,~\ldots,~\tau_{\ell}\}$ be such that $\tau_1\geq\tau_2\geq\ldots\geq\tau_{\ell}$. Define $S_i=\{u \in V_H : \bfx(u) \geq \tau_i\}$ for all $i = 1,~2,~\ldots,~\ell$. It implies that $\emptyset\subsetneq S_1 \subseteq S_2,~\ldots,~\subseteq S_{\ell} = V_H$. Let $\gamma:1\leq\gamma\leq\ell$ be largest integer such that $\tau_{\gamma} > 0$.} 
\begin{equation}\label{eq:eq4-cheeger-ineq-upper-bound-proof}
    \evalNLH_2(H)\sum_{u:u\in S^+}d_H(u)(\bfx(u))^2 = \sum_{u:u\in S^+}\bfx(u)\left(d_H(u)\bfx(u) - \sum_{v:v\in V_H\setminus\{u\}}A_H[u,v]\bfx(v)\right)
\end{equation}
Using $d_H(u)=\sum\limits_{v:v\in V_H\setminus\{u\}}A_H[u,v]$, we get
\begin{align}\label{eq:eq6-cheeger-ineq-upper-bound-proof}
    &\text{\hspace{-1.5cm}}\evalNLH_2(H)\sum_{u:u\in S^+}d_H(u)(\bfx(u))^2 \nonumber\\
    &\text{\hspace{-1cm}}=\sum_{u:u\in S^+}\bfx(u)\sum_{v:v\in V_H\setminus\{u\}}A_H[u,v]\left(\bfx(u)-\bfx(v)\right) \nonumber \\
    &\text{\hspace{-1cm}}=\sum_{u:u\in S^+}\bfx(u)\sum_{v:v\in S^+\setminus\{u\}}A_H[u,v]\left(\bfx(u)-\bfx(v)\right) + \sum_{u:u\in S^+}\bfx(u)\sum_{v:v\in V_H\setminus S^+}A_H[u,v]\left(\bfx(u)-\bfx(v)\right) \\
    &\text{\hspace{-1cm}}\geq\sum_{\{u,v\}:\{u,v\}\subseteq S^+}A_H[u,v]\left(\bfx(u)-\bfx(v)\right)^2 + \sum_{u:u\in S^+}\bfx(u)\sum_{v:v\in V_H\setminus S^+}A_H[u,v]\bfx(u)
\end{align}
where we have used the facts that $\bfx(v)\leq 0$ when $v\in V_H\setminus S^+$ and $\bfx(u)\left(\bfx(u)-\bfx(v)\right)+\bfx(v)\left(\bfx(v)-\bfx(u)\right)=\left(\bfx(u)-\bfx(v)\right)^2$. It does not affect Eq.~\eqref{eq:eq6-cheeger-ineq-upper-bound-proof} if we replace $\bfx$ values by the corresponding $\bfy$ values because $\bfx(u)$ and $\bfx(v)$ occur in this inequality only when $u,v\in S^+$. We get the following from Eq.~\eqref{eq:eq6-cheeger-ineq-upper-bound-proof} after using $\bfy(v)=0$ when $v\in V_H\setminus S^+$.
\begin{align}\label{eq:eq8-cheeger-ineq-upper-bound-proof}
    &\text{\hspace{0cm}}\evalNLH_2(H)\sum_{u:u\in S^+}d_H(u)(\bfy(u))^2 \nonumber\\
    &\text{\hspace{0cm}}\geq\sum_{\{u,v\}:\{u,v\}\subseteq S^+}A_H[u,v]\left(\bfy(u)-\bfy(v)\right)^2 + \sum_{u:u\in S^+}\sum_{v:v\in V_H\setminus S^+}A_H[u,v]\left(\bfy(u)-\bfy(v)\right)^2 \nonumber\\
    &\text{\hspace{0cm}}=\sum_{\{u,v\}:\{u,v\}\subseteq V_H}A_H[u,v]\left(\bfy(u)-\bfy(v)\right)^2
\end{align}
It implies that
\begin{equation}\label{eq:eq9-cheeger-ineq-upper-bound-proof}
    \evalNLH_2(H)\geq\frac{\sum_{\{u,v\}:\{u,v\}\subseteq V_H}A_H[u,v](\bfy(u)-\bfy(v))^2}{\sum_{u:u\in S^+}d_H(u)(\bfy(u))^2}
\end{equation}
Let $Q=\frac{\sum_{\{u,v\}:\{u,v\}\subseteq V_H}A_H[u,v](\bfy(u)-\bfy(v))^2}{\sum_{u:u\in S^+}d_H(u)(\bfy(u))^2}$. Using Cauchy-Schwarz inequality after multiplying the numerator and denominator by $\sum_{\{u,v\}:\{u,v\}\subseteq V_H}A_H[u,v](\bfy(u)+\bfy(v))^2$, we get
\begin{equation}\label{eq:eq10-cheeger-ineq-upper-bound-proof}
    Q\geq\frac{\left(\sum_{\{u,v\}:\{u,v\}\subseteq V_H}A_H[u,v]|(\bfy(u))^2-(\bfy(v))^2|\right)^2}{\left(\sum_{u:u\in S^+}d_H(u)(\bfy(u))^2\right)\left(\sum_{\{u,v\}:\{u,v\}\subseteq V_H}A_H[u,v](\bfy(u)+\bfy(v))^2\right)}
\end{equation}
We also have
{\allowdisplaybreaks
\begin{align}\label{eq:eq11-cheeger-ineq-upper-bound-proof}
    &\text{\hspace{0cm}}\sum_{\{u,v\}:\{u,v\}\subseteq V_H}A_H[u,v](\bfy(u)+\bfy(v))^2 \nonumber\\
    &\text{\hspace{0.5cm}}=\sum_{\{u,v\}:\{u,v\}\subseteq V_H}A_H[u,v]\left\{2(\bfy(u))^2+2(\bfy(v))^2-(\bfy(u)-\bfy(v))^2\right\} \nonumber\\
    &\text{\hspace{0.5cm}}=2\sum_{\{u,v\}:\{u,v\}\subseteq V_H}A_H[u,v]\left\{(\bfy(u))^2+(\bfy(v))^2\right\}-\sum_{\{u,v\}:\{u,v\}\subseteq V_H}A_H[u,v](\bfy(u)-\bfy(v))^2 \nonumber\\
    &\text{\hspace{0.5cm}}=2\sum_{\{u,v\}:\{u,v\}\subseteq V_H}A_H[u,v]\left\{(\bfy(u))^2+(\bfy(v))^2\right\} \nonumber\\
    &\text{\hspace{4.0cm}}-\frac{\sum_{\{u,v\}:\{u,v\}\subseteq V_H}A_H[u,v](\bfy(u)-\bfy(v))^2}{\sum_{u:u\in S^+}d_H(u)(\bfy(u))^2}\sum_{u:u\in S^+}d_H(u)(\bfy(u))^2 \nonumber\\
    &\text{\hspace{0.5cm}}=(2 - Q)\sum_{u:u\in S^+}d_H(u)(\bfy(u))^2
\end{align}}
Using Eq.~\eqref{eq:eq11-cheeger-ineq-upper-bound-proof} in Eq.~\eqref{eq:eq10-cheeger-ineq-upper-bound-proof}, we get
\begin{equation}\label{eq:eq12-cheeger-ineq-upper-bound-proof}
    Q(2 - Q)\geq\frac{\left(\sum_{\{u,v\}:\{u,v\}\subseteq V_H}A_H[u,v]|(\bfy(u))^2-(\bfy(v))^2|\right)^2}{\left(\sum_{u:u\in S^+}d_H(u)(\bfy(u))^2\right)^2}
\end{equation}
Either $\min\{\vol(S^+),\vol(V_H\setminus S^+)\}=\vol(S^+)$ or $\min\{\vol(S^+),\vol(V_H\setminus S^+)\}=\vol(V_H\setminus S^+)$. We complete the proof for the case when $\min\{\vol(S^+),\vol(V_H\setminus S^+)\}=\vol(S^+)$. The proof follows similarly for the other case. Let $\{\bfy(u) : u \in V_H\}=\{\tau_1,~\tau_2,~\ldots,~\tau_{\ell}\}$ be such that $\tau_1\geq\tau_2\geq\ldots\geq\tau_{\ell}=0$. Define $S_i=\{u \in V_H : \bfy(u) \geq \tau_i\}$ for all $i = 1,~2,~\ldots,~\ell$. It implies that $\emptyset\subsetneq S_1 \subsetneq S_2\subsetneq~\ldots~\subsetneq S_{\ell-1}\subsetneq S_{\ell} = V_H$ where $S_{\ell-1}=S^+$. Therefore, $\min\{\vol(S^+),\vol(V_H\setminus S^+)\}=\vol(S^+)$ implies that $\min\{\vol(S_i),\vol(V_H\setminus S_i)\}=\vol(S_i)$ for all $i:1\leq i\leq\ell-1$. We have
{\allowdisplaybreaks
\begin{align*}
    &\text{\hspace{-0.5cm}}\sum_{\{u,v\}:\{u,v\}\subseteq V_H}A_H[u,v]|(\bfy(u))^2-(\bfy(v))^2|\\
    &\text{\hspace{3.5cm}}=\sum_{\alpha=1}^{\ell-1}\sum_{\underset{\bfx(u)=\tau_{\alpha}}{u \in V_H}}\sum_{\beta=\alpha+1}^{\ell}\sum_{\underset{\bfx(v)=\tau_{\beta}}{v \in V_H}} A_H[u,v]\left(\tau_{\alpha}^2-\tau_{\beta}^2\right) \\
    &\text{\hspace{3.5cm}}=\sum_{\alpha=1}^{\ell-1}\sum_{\underset{\bfx(u)=\tau_{\alpha}}{u \in V_H}}\sum_{\beta=\alpha+1}^{\ell}\sum_{\underset{\bfx(v)=\tau_{\beta}}{v \in V_H}} A_H[u,v]\sum_{i=\alpha}^{\beta-1}\left(\tau_i^2-\tau_{i+1}^2\right)\\
    &\text{\hspace{3.5cm}}=\sum_{i=1}^{\ell-1}\sum_{\alpha=1}^{i}\sum_{\underset{\bfx(u)=\tau_{\alpha}}{u \in V_H}}\sum_{\beta=i+1}^{\ell}\sum_{\underset{\bfx(v)=\tau_{\beta}}{v \in V_H}} A_H[u,v]\left(\tau_i^2-\tau_{i+1}^2\right)\\
    &\text{\hspace{3.5cm}}=\sum_{i=1}^{\ell-1}\left(\tau_i^2-\tau_{i+1}^2\right)\sum_{u \in S_i}\sum_{v \in V_H\setminus S_i} A_H[u,v]
\end{align*}}
For all $i:0 \leq i \leq \ell - 1$, note that $\newphi(H) \leq \newphi(S_i)=\frac{\sum_{u \in S_i}\sum_{v \in V_H\setminus S_i} A_H[u,v]}{\vol_H(S_i)}$. It implies that
\begin{align*}
    &\text{\hspace{-0.5cm}}\sum_{\{u,v\}:\{u,v\}\subseteq V_H}A_H[u,v]|(\bfy(u))^2-(\bfy(v))^2|\\
    &\text{\hspace{3.5cm}}\geq\newphi(H)\sum_{i=0}^{\ell-1}\vol_H(S_i)\left(\tau_i^2-\tau_{i+1}^2\right)\\
    &\text{\hspace{3.5cm}}=\newphi(H)\left\{\vol_H(S_0)+\sum_{i=1}^{\ell-1}\vol(S_i\setminus S_{i-1})\tau_i^2\right\}\\
    &\text{\hspace{3.5cm}}=\newphi(H)\sum_{u:u\in S^+}d_H(u)(\bfy(u))^2
\end{align*}
Using this in Eq.~\eqref{eq:eq12-cheeger-ineq-upper-bound-proof}, we get
\begin{equation}\label{eq:eq15-cheeger-ineq-upper-bound-proof}
    Q(2 - Q)\geq\left(\newphi(H)\right)^2
\end{equation}
Note that $0\leq Q\leq \evalNLH_2(H)\leq 1$, and $Q(2 - Q)$ is an increasing function of $Q$ in interval $[0,1]$. This implies that $\evalNLH_2(H)(2 - \evalNLH_2(H))\geq\left(\newphi(H)\right)^2$. Therefore, we have
\[\newphi(H)\leq\sqrt{(2-\evalNLH_2(H))\evalNLH_2(H)}\]

\section{Conclusion And Open Questions}
\label{sec:conclusion}

In this work, we developed a spectral framework for analyzing general non-uniform hypergraphs. Building on Banerjee's adjacency matrix and Spiro's averaging-based diffusion, we established Cheeger-type inequalities for hypergraph conductance and derived an improved Cheeger's inequality for non-covering hypergraphs. We also showed that these bounds are tight on appropriate families of hypergraphs. Furthermore, our framework enables higher-order Cheeger inequalities and provides theoretical guarantees for Fiedler's spectral partitioning algorithm in hypergraphs. Most notably, we constructed a family of optimal hypergraph expanders that is tight for the Alon--Boppana bound.

\ignore{Our work reinforces the perspective that, unlike edges in standard $2$-graphs, which can be interpreted either as connectors or as channels for diffusion, it is more insightful to view \emph{hyperedges} primarily as mechanisms for \emph{multiway diffusion}. The \emph{averaging-based diffusion} process, which distributes mass equally among all vertices in a hyperedge, was studied by Spiro~\cite{spiro2022averaging} and later used by Kamal and Bagchi~\cite{kamal2024lovasz} to establish the Lovász--Simonovits theorem. This diffusion model can be interpreted as a ``lazy'' variant of Banerjee’s adjacency matrix, where the contribution of each edge to the matrix entries is normalized by the edge size minus $1$. The term \emph{lazy} is borrowed from the context of Markov chains, where a lazy chain remains in its current state with some probability, typically $1/2$, instead of always transitioning.

In Spiro’s formulation, a vertex diffusing mass through a hyperedge $e$ retains a fraction $1/|e|$ for itself and redistributes the remaining mass according to Banerjee’s adjacency structure. This approach stands in contrast to the diffusion models proposed by Takai, Miyauchi, Ikeda, and Yoshida~\cite{takai2020hypergraph} and Chan, Louis, Tang, and Zhang~\cite{chan2018spectral}, where mass is transferred from higher-mass vertices to lower-mass ones \emph{within} each hyperedge, depending on the current mass distribution. Their perspective is closer to a \emph{dyadic} or pairwise connectivity-based notion of diffusion, unlike Spiro’s uniform, structure-based approach.}

Our results leave the following important questions open. First, it is not yet known whether Corollary~\ref{thm:cheeger-ineq-hypergraph-cor} is \emph{tight on the upper bound} when the rank of the hypergraph is treated as a parameter. Second, in the context of \emph{expander construction}, our current method ensures that the rank of the hypergraph expander is at most the degree. A natural question is whether it is possible to construct expanders satisfying $\Upsilon_H = \Theta(d_H)$, which would yield hypergraph expanders where the largest hyperedge size grows proportionally with the vertex degree.

\ignore{Another promising direction for future work is the study of \emph{edge-dependent diffusions}, in which each hyperedge is allowed to diffuse mass according to a different, possibly non-uniform, distribution. This could lead to a richer class of diffusion dynamics with broader applications. A deeper exploration of the structure and properties of \emph{hypergraph expanders} is also likely to offer valuable theoretical and practical insights.}


\bibliographystyle{alphaurl}
\bibliography{hypergraph-diffusion}
\end{document}